\documentclass[a4paper]{article}

\usepackage[cyr]{aeguill}
\usepackage[english]{babel}
\usepackage[utf8]{inputenc}
\usepackage[T1]{fontenc}

\usepackage[a4paper,top=3cm,bottom=2cm,left=3cm,right=3cm,marginparwidth=1.75cm]{geometry}

\usepackage{amsmath}
\usepackage{graphicx}
\usepackage[colorinlistoftodos]{todonotes}
\usepackage[colorlinks=true, allcolors=blue]{hyperref}
\usepackage{amssymb}
\usepackage{amsthm}
\usepackage{bbold}
\usepackage[all]{xy}
\usepackage{dirtytalk}
\usepackage[titletoc,title]{appendix}
\usepackage{ulem}
\newtheorem{theorem}{Theorem}[section]
\newtheorem{corollary}[theorem]{Corollary}
\newtheorem{lemma}[theorem]{Lemma}
\theoremstyle{definition}

\newtheorem{definition}[theorem]{Definition}
\newtheorem{example}[theorem]{Example}

\newtheorem{proposition}[theorem]{Proposition}
\newtheorem{remark}[theorem]{Remark}

\makeatletter
\newcommand{\oset}[2]{{\mathpalette\o@set{{#1}{#2}}}}
\newcommand{\o@set}[2]{\o@@set{#1}#2}
\newcommand{\o@@set}[3]{%
  \vbox{\offinterlineskip
    \ialign{\hfil##\hfil\cr
      $\m@th\o@set@demote{#1}#2$\cr
      \noalign{\vskip0.2pt}
      $\m@th#1#3$\cr
    }%
  }%
}
\newcommand{\o@set@demote}[1]{%
  \ifx#1\displaystyle\scriptstyle\else
  \ifx#1\textstyle\scriptstyle\else
  \scriptscriptstyle\fi\fi
}
\makeatother

\newcommand{\owave}[1]{\oset{\rotatebox{180}{\uwave{\hspace{\widthof{\ensuremath{#1}}}}}}{#1}}

\title{Fourier decay of equilibrium states on hyperbolic surfaces}
\author{Gaétan Leclerc}

\date{ \vspace{-2em} }

\begin{document}

\maketitle

\begin{abstract}
Let $\Gamma$ be a (convex-)cocompact group of isometries of the hyperbolic space $\mathbb{H}^d$, let $M := \mathbb{H}^d/\Gamma$ be the associated hyperbolic manifold, and consider a real valued potential $F$ on its unit tangent bundle $T^1 M$. Under a natural regularity condition on $F$, we prove that the associated $(\Gamma,F)$-Patterson-Sullivan densities are stationary measures with exponential moment for some random walk on $\Gamma$. As a consequence, when $M$ is a surface, the associated equilibrium state for the geodesic flow on $T^1 M$ exhibit \say{Fourier decay}, in the sense that a large class of oscillatory integrals involving it satisfies power decay. It follows that the non-wandering set of the geodesic flow on convex-cocompact hyperbolic surfaces has positive Fourier dimension, in a sense made precise in the appendix.
\end{abstract}

\section{Introduction}

\subsection{State of the art}

Since the early work of Dolgopyat on the decay of correlation of Anosov flows \cite{Do98}, we know that the rate of mixing of hyperbolic flows (with respect to some \textbf{equilibrium states}) may be linked with the spectral properties of \say{twisted transfer operators}. This idea has been widely used and generalized: see, for example, \cite{St01}, \cite{BV05}, \cite{St11}, and \cite{DV21}, to only name a few related work. In concrete terms, exponential mixing can be reduced to exhibiting enough cancellations in sums of exponentials. It turns out that this (rather complicated) process can be sometimes simplified if one knows that the Fourier transform of those equilibrium states exhibit Fourier decay. This idea has been explored and discussed in Li's work on stationary measures \cite{Li20}, as well as in \cite{MN20}, and more recently in a preprint by Khalil \cite{Kh23}. This connection with the exponential mixing of dynamical systems has sparked recent interest in studying the behavior of the Fourier transform of measures. Historically, this was motivated by understanding sets of unicity for Fourier series \cite{KS64}, which lead us to discover that the Fourier properties of a measure may be used to study the arithmetic properties of its support. This idea is encoded in the notion of \say{Fourier dimension}: see for exemple a recent preprint of Fraser \cite{Fr22} (introducing a new notion of \say{Fourier dimension spectrum}) and the references therein. Let us introduce this notion. \\

The Fourier dimension is better understood if we first recall a well known formula for the Hausdorff dimension. If $E \subset \mathbb{R}^d$ is a compact subset of some euclidean space, a corollary (see for example \cite{Ma15}) of a lemma by Frostmann \cite{Fro35} yields the following identity:
$$ \dim_H E = \sup\left\{ \alpha \in [0,d] \ | \ \exists \mu \in \mathcal{P}(E), \ \int_{\mathbb{R}^d} |\widehat{\mu}(\xi)|^2 |\xi|^{\alpha-d} d\xi < \infty \right\}, $$
where $\dim_H E$ is the Hausdorff \say{fractal} dimension of $E$, $\mathcal{P}(E)$ is the set of all (borel) probability measures supported on $E$, and where $\widehat{\mu} : \mathbb{R}^d \rightarrow \mathbb{C}$, given by
$$ \widehat{\mu}(\xi) := \int_{\mathbb{R}^d} e^{-2i \pi x \cdot \xi} d\mu(x), $$
is the Fourier transform of the measure $\mu \in \mathcal{P}(E)$. The condition on the measure in the supremum can be though as a decay condition \say{on average}. In particular, the inner integral is finite if $\widehat{\mu}(\xi)$ decays like $|\xi|^{-\alpha/2-\varepsilon}$ for large $\xi$. With this in mind, the following notion is quite natural to introduce: we define the Fourier dimension of $E \subset \mathbb{R}^d$ by the formula
$$ \dim_F E := \sup\left\{ \alpha \in [0,d] \ | \ \exists \mu \in \mathcal{P}(E), \exists C \geq 1, \forall \xi \in \mathbb{R}^d \setminus \{0\}, \ |\widehat{\mu}(\xi)| \leq C |\xi|^{-\alpha/2} \right\} .$$
While it is clear that $0 \leq \dim_F E \leq \dim_H E \leq d$, we do not always have equality between the two notions. For example, it is well known that the triadic Cantor set has Hausdorff dimension  $\ln2/\ln3$, but has Fourier dimension $0$. In fact, exhibiting deterministic (fractal) sets with positive Fourier dimension is quite a challenge. One of the earliest examples of deterministic (fractal) set with positive Fourier dimension was discovered by Kaufmann in 1980 \cite{Ka80}, involving sets related to continued fractions. Kaufmann's method was optimised by Queffelec and Ramare in \cite{QR03}, and was more recently generalized by Jordan and Sahlsten \cite{JS16}. A year later, Kaufmann \cite{Ka81} found deterministic examples of (fractal) \emph{Salem sets} in $\mathbb{R}$, that is, sets $E$ satisfying $\dim_H E = \dim_F E$. This example is related to diophantine approximations. The construction was generalized in $\mathbb{R}^2$ in \cite{Ha17}, and then in $\mathbb{R}^d$ in \cite{FH23}. Let us mention that there is a whole bibliography on random constructions of Salem sets: the interested reader may look at the references in \cite{FH23}. Returning to the study of sets with positive Fourier dimension and of sets of unicity, let us also mention that a there has been a lot of interest for Cantor sets appearing from linear IFS. Related work includes \cite{LS19a}, \cite{LS19b}, \cite{So19}, \cite{Br19}, \cite{VY20}, and \cite{Ra21}.  \\

In 2017, Bourgain and Dyalov \cite{BD17} introduced a new method to prove positivity of the Fourier dimension for some sets. The paper takes place in a dynamical context. If we fix some Schottky group $\Gamma < \text{PSL}(2,\mathbb{R})$, then $\Gamma$ acts naturally on $\mathbb{R}$. One can prove that there exists a cantor set  $\Lambda_\Gamma \subset \mathbb{R}$, called a limit set, which is invariant by $\Gamma$. On this limit set, there is a family of natural probability measures associated with the dynamics that are called Patterson-Sullivan measures. Using results from additive combinatorics, more specifically a \say{sum-product phenomenon}, Bourgain and Dyatlov managed to prove power decay for the Fourier transform of those probability measures. In particular, the limit set $\Lambda_\Gamma \subset \mathbb{R}$ has positive Fourier dimension. An essential feature of $\Gamma$ was the \emph{nonlinearity} of the dynamics. \\

The method introduced in this paper inspired numerous generalizations, beginning with a paper of Li, Naud and Pan \cite{LNP19} proving power decay for Patterson-Sullivan measures over (Zariski dense) Fuschian Schottky groups $\Gamma < \text{PSL}(2,\mathbb{C})$. In this paper, Li proves that such measures may be seen as \textbf{stationary measures with finite exponential moment}, which allows them to use several results from the topic of random walks on groups. From this, they obtain positivity of the Fourier dimension of the associated limit set $\Lambda_\Gamma \subset \mathbb{C}$. \\

From there, at least two different directions exists for generalization. The first one is to notice that Patterson-Sullivan measures are equilibrium states for some hyperbolic dynamical system. A natural generalization is then given by Sahlsten and Steven in \cite{SS20}: for one dimensional and \say{totally nonlinear} IFS, one can show that any equilibrium state exhibit Fourier decay. In particular, Sahslten and Steven obtain positive Fourier dimension for a large class of \say{nonlinear} Cantor sets. This paper use the method introduced by Bourgain-Dyatlov, and is also inspired by previous techniques appearing in \cite{JS16}. See also \cite{ARW20} for some related work on nonlinear IFS and pointwise normality. Some complementary remarks on the work of Sahlsten and Steven may be found in \cite{Le22}. Past the one-dimensionnal setting, it was proved by the author in \cite{Le21} that the same results hold true in the context of hyperbolic Julia sets in the complex plane. Some decay results are also true, in the unstable direction, for equilibrium states of (sufficiently bunched) nonlinear solenoids \cite{Le23}. \\

A second natural direction to look at is for result concerning stationary measures with exponential moment for random walks on groups. Li proved in \cite{Li18} and \cite{Li20} several Fourier decay results in the context of random walks over $SL_n(\mathbb{R})$ (a crucial property of this group in the proofs is its splitness). Past the split setting, further results seems difficult to achieve. 

\subsection{Our setting of interest}

In this paper, we are interested in studying the Fourier properties of equilibrium states for the geodesic flow on convex-cocompact surfaces of constant negative curvature. More details on our setting will be explained during the paper, but let's quickly introduce the main objects at play. A usefull reference is \cite{PPS15}. \\

We work on hyperbolic manifolds, that is, a riemannian manifold $M$ that may be written as $M = \mathbb{H}^d/\Gamma$, where $\mathbb{H}^d$ is the hyperbolic space of dimension $d$, and where $\Gamma$ is a (non-elementary, discrete, without torsion, orientation preserving) group of isometries of $\mathbb{H}^d$. The geodesic flow $\phi = (\phi_t)_{t \in \mathbb{R}}$ acts on the unit tangent bundle of $M$, denoted by $T^1 M$. We say that a point $v \in T^1 M$ is wandering for the flow if there exists an open neighborhood $U \subset T^1 M$ of $v$, and a positive number $T>0$ such that:
$$ \forall t >T, \ \phi_t(U) \cap U = \emptyset .$$
The set of non-wandering points for $\phi$, denoted by $NW(\phi) \subset T^1 M$, is typically \say{fractal} and is invariant by the geodesic flow. We will work under the hypothesis that the group $\Gamma$ is convex-cocompact, which exactly means that $\text{NW}(\phi)$ is supposed compact. In particular, \emph{the case where $M$ is itself compact is authorized}. Under this condition, the flow $\phi$ restricted to $NW(\phi)$ is Axiom A. \\

In this context, for any choice of Hölder regular potential $F:T^1 M \rightarrow \mathbb{R}$, and for any probability measure $m \in \mathcal{P}(T^1 M)$ (the set of borel probability measures on $T^1 M$) invariant by the geodesic flow, one can consider the \emph{metric pressure} associated to $m$, defined by:
$$ P_{\Gamma,F}(m) = h_m(\phi) + \int_{\text{NW}(\phi)} F dm ,$$
where $h_m(\phi)$ denotes the entropy of the time-1 map of the geodesic flow with respect to the measure $m$. Notice that any probability measure invariant by the geodesic flow must have support included in the non-wandering set of $\phi$.
The \emph{topological pressure} is then defined by
$$ P(\Gamma,F) := \sup_m P_{\Gamma,F}(m) ,$$
where the sup is taken over all the $\phi$-invariant probability measures $m$. Those quantities generalize the variationnal principle for the topological and metric entropy (that we recover when $F=0$). 
It is well known that this supremum is, in fact, a maximum: see for example \cite{BR75} or $\cite{PPS15}$.

\begin{theorem}
Let $\Gamma$ be convex-cocompact, $M := \mathbb{H}^d/\Gamma$, and $F:T^1 M \rightarrow \mathbb{R}$ be a Hölder regular potential. Then there exists a unique probability measure $m_F$ invariant by $\phi$ such that $P_{\Gamma,F}(m_F) = P({\Gamma,F})$. This measure is called the equilibrium state associated to $F$ and its support is the non-wandering set of the geodesic flow. When $F=0$, $m_F$ is the measure of maximal entropy.
\end{theorem}

Theorem 6.1 in \cite{PPS15} also gives us a description of equilibrium states. To explain it, recall that the \emph{Hopf coordinates} allows us to identify $T^1 \mathbb{H}^d$ with $\partial_\infty \mathbb{H}^{d} \times \partial_\infty \mathbb{H}^{d} \times \mathbb{R}$, where $\partial_\infty \mathbb{H}^d $ denotes the \emph{ideal boundary} of the hyperbolic space (diffeomorphic to a sphere in our context). The measure $m_F$ lift into a $\Gamma$-invariant measure $\tilde{m}_F$ on $T^1 \mathbb{H}^d$, which can then be studied in these coordinates. The interesting remark is that $\tilde{m}_F$ may be seen as a product measure, involving what we call $(\Gamma,F)$-Patterson-Sullivan densities, which are generalization of the usual Patterson-Sullivan probability measures. More precisely, there exists $\mu_F$ and $\mu_{F}^\iota$, two Patterson-Sullivan densities supported on the ideal boundary $\partial_\infty \mathbb{H}^d$, such that one may write (in these Hopf coordinates):

$$ d\tilde{m}_F(\xi,\eta,t) = \frac{d\mu_F(\xi) \otimes d\mu_{F}^\iota(\eta) \otimes dt}{D_{F}(\xi,\eta)^2} ,$$
where $D_F$ is the \say{potential gap} (or gap map), that we will define later. More details on Patterson-Sullivan densities can be found in section 2 (which will be devoted to recalling various preliminary results). 
Since the Hopf coordinates are smooth on $\mathbb{H}^d$, we see that one may reduce Fourier decay for $m_F$ to proving Fourier decay for Patterson-Sullivan densities. This reduction is the content of section 4. Then, to prove Fourier decay for those measures, several possibilities exists. With our current techniques, this may only be achieved when $d=2$, so that Patterson-Sullivan densities are supported on the circle. \\

The first possibility would be to use the fact that, in this low dimensionnal context, there exists a coding of the dynamics of the group $\Gamma$ on the ideal boundary: see for exemple \cite{BS79} or \cite{AF91}. Using these, one should be able to get Fourier decay for Patterson-Sullivan densities by adapting the proof of Bougain and Dyatlov in \cite{BD17}. The second possibility would be to adapt the argument found in Li's appendix \cite{LNP19} to prove that Patterson-Sullivan densities are actually stationary measures with exponential moment (for a random walk on $\Gamma$). Since in dimension 2, isometries of $\mathbb{H}^2$ may be seen as elements of $\text{SL}_2(\mathbb{R})$, one could then apply Li's work \cite{Li20} to get Fourier decay. This is the strategy that we choose to follow in section 3. Finally, let us enhance the fact that we are only able to work under a regularity condition (R) (see definition 2.13) that ensure Hölder regularity for our measures of interest. We now state our main results.

\begin{theorem}[Compare Theorem 3.2]
Let $\Gamma$ be a convex-cocompact group of isometries of $\mathbb{H}^d$, and let $F : T^1(\mathbb{H}^d/\Gamma) \rightarrow \mathbb{R}$ be a Hölder potential satisfying (R). Let $\mu \in \mathcal{P}(\partial_\infty \mathbb{H}^d)$ be a $(\Gamma,F)$ Patterson-Sullivan density. Then there exists $\nu \in \mathcal{P}(\Gamma)$ with exponential moment such that $\mu$ is $\nu$-stationary and such that the support of $\nu$ generates $\Gamma$.
\end{theorem}

Theorem 1.2 is our main technical result. The strategy is inspired by the appendix of \cite{LNP19}, but in our setting, some additionnal difficulties appear since the potential may be non-zero. For example, the proof of Lemma A.12 in \cite{LNP19} fails to work in our context. Our main idea to replace this lemma is to do a carefull study of the action of $\Gamma$ on the sphere at infinity: we will be particularly interested in understanding its contractions properties. This is the content of section 2. The proof of Theorem 1.2 is in section 3. Once this main technical result is proved, one can directly use the work of Li \cite{Li20} and get:

\begin{corollary}[\cite{Li20}, Theorem 1.5]

Let $\Gamma$ be a convex-cocompact group of isometries of $\mathbb{H}^2$, and let $F : T^1(\mathbb{H}^2/\Gamma) \rightarrow \mathbb{R}$ be a Hölder potential satisfying (R). Let $\mu \in \mathcal{P}(\Lambda_\Gamma)$ be a $(\Gamma,F)$ Patterson-Sullivan density. There exists $\varepsilon>0$ such that the following hold. Let $R \geq 1$ and let $\chi : \partial_\infty \mathbb{H}^2 \simeq \mathbb{S}^1 \rightarrow \mathbb{R}$ be a $\alpha$-Hölder map supported on some compact $K$. Then there exists $C \geq 1$ such that, for any $C^2$ function $\varphi : \partial_\infty \mathbb{H}^2 \rightarrow \mathbb{R}$ such that $\|\varphi\|_{C^2} + (\inf_K |\varphi'|)^{-1} \leq R$, we have:

$$ \forall s \in \mathbb{R}^*, \ \left|\int_{\partial_\infty \mathbb{H}^2} e^{i s \varphi} \chi d\mu \right| \leq \frac{C}{ |s|^{\varepsilon}} .$$
\end{corollary}

Using the previous Corollary 1.3 and using the Hopf coordinates, we can conclude Fourier decay for equilibrium states on convex-cocompact hyperbolic surfaces. The proof is done in section 4.

\begin{theorem}[Compare Theorem 4.5]
Let $\Gamma$ be a convex-cocompact group of isometries of $\mathbb{H}^2$, and let $F : T^1(\mathbb{H}^2/\Gamma) \rightarrow \mathbb{R}$ be a Hölder potential satisfying (R). Let $m_F$ be the associated equilibrium state. There exists $\varepsilon>0$ such that the following holds. Let $\chi : T^1\mathbb{H}^2 \rightarrow \mathbb{R}$ be a Hölder map supported on a compact neighborhood of some point $v_o \in T^1 \mathbb{H}^2$, and let $\varphi : T^1 \mathbb{H}^2 \rightarrow \mathbb{R}^3$ be a $C^2$ local chart containing the support of $\chi$. There exists $C\geq 1$ such that:

$$ \forall \zeta \in \mathbb{R}^3\setminus \{0\}, \ \left|\int_{NW(\phi)} e^{i \zeta \cdot \varphi(v)} \chi(v) d m_F(v) \right| \leq \frac{C}{|\zeta|^{\varepsilon}},$$
where $\zeta \cdot \zeta'$ and $|\zeta|$ denotes the euclidean scalar product and the euclidean norm on $\mathbb{R}^3$. In other word, the pushforward measure $\varphi_*(\chi d m_F) \in \mathcal{P}(\mathbb{R}^3)$ exhibit power Fourier decay.
\end{theorem}

\begin{remark}
We will see in section 4 that the argument to prove Theorem 1.4 from Corollary 1.3 is fairly general. In particular, if one is able to prove Fourier decay for $(\Gamma,F)$-Patterson-Sullivan densities in some higher dimensionnal context, this would prove Fourier decay for equilibrium states in higher dimensions. For example, \cite{LNP19} precisely proves Fourier decay for Patterson-Sullivan densities with the potential $F=0$ when $\Gamma < \text{PSL}(2,\mathbb{C})$ is a Zariski-dense Kleinian Schottky group. This yields power decay for the measure of maximal entropy on $M := \mathbb{H}^3/\Gamma$ in this context. 
\end{remark}

\begin{remark}
With our result in mind, it is natural to try to give some sense to the sentence \say{$\dim_F NW(\phi) >0 $}. The problem is that the notion of Fourier dimension is not well defined on manifolds. In the appendix, we suggest some natural notions of Fourier dimensions for sets living in a manifold, in particular a notion of \emph{lower Fourier dimension}, that measure \say{persistence of the positivity of the Fourier dimension under deformations}. The sentence is then made rigorous in Remark A.7 and Example A.23.
\end{remark}

\subsection{Acknowledgments}
I would like to thank my PhD advisor, Frederic Naud, for pointing out to me the existing bibliography on Patterson-Sullivan densities and for encouraging me to work in the context of hyperbolic surfaces. I would also like to thank Jialun Li for explaining to me his work on Fourier decay for stationary measures with exponential moment in the context of split groups.  Finally, I would like to thank Paul Laubie for some very stimulating discussions. 
\section{Preliminaries}

\subsection{Moebius transformations preserving the unit ball}

In this first paragraph we recall well known properties of Moebius transformations. Usefull references for the study of such maps are \cite{Be83}, \cite{Ra06} and \cite{BP92}. The group of all Moebius transformations of $\mathbb{R}^d \cup \{\infty\}$ is the group generated by the inversion of spheres and by the reflexions. This group contains dilations and rotations. Denote by $\text{Mob}(B^b)$ the group of all the Moebius transformations $\gamma$ such that $\gamma$ preserves the orientation of $\mathbb{R}^d$, and such that $\gamma(B^d)=B^d$, where $B^d$ denotes the open unit ball in $\mathbb{R}^d$. These maps also acts on the unit sphere $\mathbb{S}^{d-1}$. These transformations can be put in a \say{normal form} as follows.

\begin{lemma}[\cite{Ra06}, page 124]
Define, for $b \in B^d$, the associated \say{hyperbolic translation} by:
$$ \tau_b(x) = \frac{(1-|b|^2) x + (|x|^2 + 2 x \cdot b + 1) b}{|b|^2 |x|^2 + 2 x \cdot b + 1} .$$
Then $\tau_b \in \text{Mob}(B^d)$. Moreover, for every $\gamma \in \text{Mob}(B^d)$, $\tau_{\gamma(0)}^{-1} \gamma  \in SO(d,\mathbb{R})$. 
\end{lemma}

It follows that the distortions of any Moebius transformation $\gamma \in \text{Mob}(B^d)$ can be understood by studying the distortions of hyperbolic translations. The main idea is the following: if $\gamma(o)$ is close to the unit sphere, then $\gamma$ contracts strongly on a large part of the sphere. Let us state a quantitative statement:

\begin{lemma}[First contraction lemma]

Let $\gamma \in \text{Mob}(B^d)$. Suppose that $|\gamma(o)| \geq c_0 > 0$. Denote by $x_\gamma^m := \gamma(o)/|\gamma(o)|$, and let $\varepsilon_\gamma := 1-|\gamma(o)|$. Then: 
\begin{enumerate}
\item There exists $c_1,c_2>0$ that only depends on $c_0$ such that
$$ \forall x \in \mathbb{S}^{d-1}, \ |x - x^m_\gamma| \geq c_1 \varepsilon_\gamma^2 \Longrightarrow |\gamma^{-1} x - \gamma^{-1} x^m_\gamma | \geq c_2 .$$
\item For all $c \in (0,1)$, there exists $C \geq 1$ and a set $A_\gamma \subset \mathbb{S}^{d-1}$ such that $\text{diam}(A_\gamma) \leq  C \varepsilon_\gamma$ and such that:
$$ \forall x \in \mathbb{S}^{d-1} \setminus A_\gamma, \ |\gamma(x)-x_\gamma^m| \leq c \varepsilon_\gamma .$$
\end{enumerate}
\end{lemma}

\begin{proof}
Let $\gamma \in \text{Mob}(B^{d})$. Since $\gamma = \tau_{\gamma(o)} \Omega$ for some $\Omega \in SO(n,\mathbb{R})$, we see that we may suppose $\gamma = \tau_b$ for some $b \in B^d$. Without loss of generality, we may even choose $b$ of the form $\beta e_d$, where $e_d$ is the d-th vector of the canonical basis of $\mathbb{R}^d$, and where $\beta = |\gamma(o)| \in [c_0,1[$. Denote by $\pi_d$ the projection on the $d-th$ coordinate. We find:
$$ \forall x \in \mathbb{S}^{d-1}, \ \pi_d \tau_b(x) = \frac{(1+\beta^2) x_d + 2\beta}{2 \beta x_d + (1+\beta^2)} =: \varphi(x_d).$$
The function $\varphi$ is continuous and increasing on $[-1,1]$, and fixes $\pm 1$. Computing its value at zero gives $ \varphi(0) = \frac{2 \beta}{1+\beta^2} \geq 1 - \varepsilon_\gamma^2 $,
which proves the first point. Computing its value at $-\beta$ gives $\varphi(-\beta) = \beta$, which (almost) proves the second point. The second point is proved rigorously by a direct computation, noticing that
$$ 1-\varphi(-1+C \varepsilon_\gamma) = \frac{\varepsilon_\gamma}{C} \frac{1-C \varepsilon_\gamma/2}{1- (1-1/(2C))\varepsilon_\gamma} \leq \varepsilon_\gamma/C .$$
\end{proof}
Finally, let us recall a well known way to see $\text{Mob}(B^d)$ as a group of matrices.
\begin{lemma}
Let $q : \mathbb{R}^{d+1} \rightarrow  \mathbb{R}^{d+1}$ be the quadratic form $q(t,\omega) := -t^2 + \sum_i \omega_i^2$  on $\mathbb{R}^{d+1}$. We denote by $SO(d,1)$ the set linear maps with determinant one that preserves $q$. Let $H := \{ (t,\omega) \in \mathbb{R} \times \mathbb{R}^d \ , \ q(t,\omega)=-1 \ , t > 0\}.$ Define the stereographic projection $\zeta : B^d \rightarrow H$ by $\zeta(x) := \left( \frac{1+|x|^2}{1-|x|^2}, \frac{2x}{1-|x|^2} \right)$. Then, for any $\gamma \in \text{Mob}(B^d)$, the map $\zeta \gamma \zeta^{-1} : H \rightarrow H$ is the restriction to $H$ of an element of $SO(d,1)$.
\end{lemma}

\begin{proof}

It suffices to check the lemma when $\gamma$ is a rotation or a hyperbolic translation. A direct computation shows that, when $\Omega \in SO(d,\mathbb{R})$, then $\zeta \Omega \zeta^{-1}$ is a rotation leaving invariant the $t$ coordinate, and so it is trivially an element of $SO(d,1)$. We now do the case where $\gamma = \tau_{\beta e_d}$ is a hyperbolic translation. We denote by $x$ the variable in $B^d$ and $(t,\omega)$ the variables in $\mathbb{R}^{d+1}$. The expression $\zeta(x) = (t,\omega)$ gives
$$ \frac{1+|x|^2}{1-|x|^2} = t, \quad \frac{2x}{1-|x|^2} = \omega, \ \text{and} \quad \zeta^{-1}(t,\omega) = \frac{\omega}{1+t} .$$
For $\alpha \in \mathbb{R}$, denote $s_\alpha := \sinh(\alpha)$ and $c_\alpha := \cosh(\alpha)$. There exists $\alpha$ such that $\beta = s_\alpha /(c_\alpha+1) = (c_\alpha-1)/s_\alpha$. For this $\alpha$, we also have $\beta^2 = (c_\alpha-1)/(c_\alpha+1)$, and $1-\beta^2 = 2/(c_\alpha+1)$. Now, we see that
$$ \tau_{\beta e_d}(x) = \frac{(1-\beta^2) x + (|x|^2 + 2 x_d \beta + 1)  \beta e_d}{\beta^2 |x|^2 + 2 x_d \beta + 1}  = \frac{\frac{2}{c_\alpha+1} x + (|x|^2 + 2 x_d \frac{c_\alpha-1}{s_\alpha} + 1) \frac{s_\alpha}{c_\alpha+1} e_d}{\frac{c_\alpha-1}{c_\alpha+1} |x|^2 + 2 x_d \frac{s_\alpha}{c_\alpha+1}  + 1} $$
$$ = \frac{2x + (s_\alpha(1+|x|^2) + 2x_d(c_\alpha-1) ) e_d }{1-|x|^2+ c_\alpha(1+|x|^2) + 2 s_\alpha x_d} = \frac{\omega + \left(s_\alpha t + (c_\alpha -1) \omega_d \right)e_d}{1+ c_\alpha t + s_\alpha \omega_d} $$
$$ = \zeta^{-1}\left( c_\alpha t + s_\alpha \omega_d , \omega + \left( s_\alpha t + (c_\alpha-1)\omega_d \right)e_d \right) ,$$
and so $\zeta \tau_{\beta e_d} \zeta^{-1}(t,\omega) = \left( c_\alpha t + s_\alpha \omega_d \ , \ \omega + \left( s_\alpha t + (c_\alpha-1)\omega_d \right)e_d \right)  $ is indeed linear in $(t,\omega)$. In this form, checking that $\zeta \tau_{\beta e_d} \zeta^{-1} \in SO(d,1)$ is immediate.
\end{proof}

\begin{remark}
From now on, we will allow to directly identify elements of $\text{Mob}(B^d)$ with matrices in $SO(d,1)$. (By continuity of $\gamma \mapsto \zeta \gamma \zeta^{-1}$, we even know that thoses matrices lies in $SO_0(d,1)$, the connected component of the identity in $SO(d,1)$). 
It follows from the previous explicit computations that, for any matrix norm on $SO(d,1)$, and for any $\gamma \in \text{Mob}(B^d)$ such that $|\gamma(o)| \geq c_0$, there exists $C_0$ only depending on $c_0$ such that  $$ \|\gamma\| \leq C_0 \varepsilon_\gamma^{-1}.$$
\end{remark}

\subsection{The Gibbs cocycle}

In this paragraph, we introduce our geometric setting. For an introduction to geometry in negative (non-constant) curvature, the interested reader may refer to \cite{BGS85}, or to the first chapters of \cite{PPS15}. \\

Let $M = \mathbb{H}^d/\Gamma$ be a hyperbolic manifold of dimension $d$, where $\Gamma \subset Iso^+(\mathbb{H}^d)$ denotes a non-elementary and discrete group of isometries of the hyperbolic space without torsion that preserves the orientation. Let $T^1 M$ denotes the unit tangent bundle of $M$, and denote by $p:T^1 M \rightarrow M$ the usual projection. The projection lift to a $\Gamma$-invariant map $\tilde{p}:T^1 \mathbb{H}^d \rightarrow \mathbb{H}^d$. Fix $F : T^1 M \rightarrow \mathbb{R}$ a Hölder map: we will call it a potential. The potential $F$ lift to a $\Gamma$-invariant map $\tilde{F} : T^1 \mathbb{H}^d \rightarrow \mathbb{R}$. All future constants appearing in the paper will implicitely depend on $\Gamma$ and $F$. \\

To be able to do use our previous results about Moebius transformations, we will work in the conformal ball model for a bit. In this model, we can think of $\mathbb{H}^d$ as being the unit ball $B^d$ equipped with the metric $ds^2 := \frac{4\|dx\|^2}{(1-\|x\|^2)^2}$. The ideal boundary $\partial_\infty \mathbb{H}^d$  (see \cite{BGS85} for a definition) of $\mathbb{H}^d$ is then naturally identified with $\mathbb{S}^{d-1}$, and its group of orientation-preserving isometries with $\text{Mob}(B^d)$. On the ideal boundary, there is a natural family of distances $(d_x)_{x \in \mathbb{H}^d}$ called visual distances (seen from $x$), defined as follow:
$$ d_x(\xi,\eta) := \lim_{t \rightarrow +\infty}  \exp\left( -\frac{1}{2} \left(  d(x,\xi_t) + d(x,\eta_t) - d(\xi_t,\eta_t) \right) \right) \in [0,1], $$
where $\xi_t$ and $\eta_t$ are any geodesic rays ending at $\xi$ and $\eta$. To get the intuition behind this quantity, picture a finite tree with root $x$ and think of $\xi$ and $\eta$ as leaves in this tree. 

\begin{lemma}[\cite{PPS15} page 15 and \cite{LNP19} lemma A.5]
The visual distances are all equivalent and induces the usual euclidean topology on $\mathbb{S}^{d-1} \simeq \partial_\infty \mathbb{H}^d$. More precisely:
$$ \forall x,y \in \mathbb{H}^d, \forall \xi,\eta \in \partial_\infty \mathbb{H}^d, \  \ e^{-d(x,y)} \leq \frac{d_x(\xi,\eta)}{d_y(\xi,\eta)} \leq e^{d(x,y)} .$$
In the ball model, the visual distance from the center of the ball is the sine of (half of) the angle.
\end{lemma}

The sphere at infinity $\partial_\infty \mathbb{H}^d$ takes an important role in the study of $\Gamma$. 
For any $x \in \mathbb{H}^d$, the orbit $\Gamma x$ accumulates on $\mathbb{S}^{d-1}$ (for the euclidean topology) into a (fractal) \emph{limit set} denoted $\Lambda_\Gamma$. This limit set is independant of $x$. We will denote by $\text{Hull}(\Lambda_\Gamma)$ the convex hull of the limit set: that is, the set of points $x \in \mathbb{H}^d$ such that $x$ is in a geodesic starting and finishing in $\Lambda_\Gamma$. Since $\Gamma$ acts naturally on $\Lambda_\Gamma$, $\Gamma$ acts on $\text{Hull}(\Lambda_\Gamma)$. Without loss of generality, we can assume that $o \in \text{Hull}( \Lambda_\Gamma)$, and we will do from now on. We will say that $\Gamma$ is convex-cocompact if $\Gamma$ is discrete, without torsion and if $\text{Hull}( \Lambda_\Gamma) / \Gamma$ is compact. In particular, in this paper, we allow $M$ to be compact. \\

We will suppose througout the paper that $\Gamma$ is convex cocompact. In this context, the set $\Omega \Gamma = p^{-1}( \text{Hull} \Lambda_\Gamma / \Gamma ) \subset T^1 M$ is compact, and it follows that $\sup_{\Omega \Gamma} |F| < \infty$. In particular, $\tilde{F}$ is bounded on $\tilde{p}^{-1}( \text{Hull}(\Lambda_\Gamma))$, which is going to allow us to get some control over line integrals involving $F$. Recall the notion of line integral in this context:
if $x,y \in \mathbb{H}^d$ are distinct points, then there exists a unique unit speed geodesic joining $x$ to $y$, call it $c_{x,y}$. We then define:
$$ \int_x^y \tilde{F} := \int_0^{d(x,y)} \tilde{F}(\dot{c}_{x,y}(s)) ds .$$
Beware that if $\tilde{F}(-v) \neq \tilde{F}(v)$ for some $v \in T^1 M$, then $\int_x^y \tilde{F}$ and $\int_y^x \tilde{F}$ might not be equal. \\
We are ready to introduce the \emph{Gibbs cocycle} and recall some of its properties.

\begin{definition}[\cite{PPS15}, page 39]
The following \say{Gibbs cocycle} $C_F : \partial_\infty \mathbb{H}^d \times \mathbb{H}^d \times \mathbb{H}^d \rightarrow \mathbb{R}$ is well defined and continuous:
$$ C_{F,\xi}(x,y) := \underset{t \rightarrow + \infty}{\lim} \left( \int_y^{\xi_t} \tilde{F} - \int_x^{\xi_t} \tilde{F}  \right) $$
where $\xi_t$ denotes any geodesic converging to $\xi$.
\end{definition}

\begin{remark}
Notice that if $\xi$ is the endpoint of the ray joining $x$ to $y$, then $$C_{F,\xi}(x,y) = - \int_x^y \tilde{F}.$$
\end{remark}
For $A \subset \mathbb{H}^d$, we call \say{shadow of $A$ seen from $x$} the set $\mathcal{O}_x A$ of all $\xi \in \partial_\infty \mathbb{H}^d$ such that the geodesic joining $x$ to $\xi$ intersects $A$. (The letter $\mathcal{O}$ stands for \say{Ombre} in french.)

\begin{proposition}[\cite{PPS15}, prop 3.4 and 3.5]
We have the following estimates on the Gibbs cocycle.

\begin{enumerate}
\item For all $R>0$, there exists $C_0>0$ such that for all $\gamma \in \Gamma$ and for all $\xi \in \mathcal{O}_o B(\gamma(o),R)$ in the shadow of the (hyperbolic) ball $B(\gamma(o),R)$ seen from $o$, we have:
$$ \left| C_{F,\xi}(o,\gamma(o)) + \int_o^{\gamma(o)} \tilde{F} \right| \leq C_0 $$
\item 
There exists $\alpha \in (0,1)$ and $C_0>0$ such that, for all $\gamma \in \Gamma$ and for all $\xi,\eta \in \Lambda_\Gamma$ such that $d_o(\xi,\eta) \leq e^{-d(o,\gamma(o))-2}$,
$$ \left| C_{F,\xi}(o,\gamma(o)) - C_{F,\eta}(o,\gamma(o)) \right| \leq C_0 e^{ \alpha d(o,\gamma(o))} d_o(\xi,\eta)^{\alpha}. $$
(The hypothesis asking $\xi,\eta$ to be very close can be understood as an hypothesis asking for the rays $[o,\xi[,[o,\eta[$ and $[\gamma(o),\xi[,[\gamma(o),\eta[$ to be close. This way, we can use the Hölder regularity of $\tilde{F}$ to get some control.)
\end{enumerate}

\end{proposition}

\subsection{Patterson-Sullivan densities}

In this paragraph, we recall the definition of $(\Gamma,F)$ Patterson-Sullivan densities, and we introduce a regularity condition. To begin with, recall the definition and some properties of the critical exponent of $(\Gamma,F)$.

\begin{definition}[\cite{PPS15}, Lemma 3.3]
Recall that $F$ is supposed Hölder, and that $\Gamma$ is convex-cocompact. The critical exponent of $(\Gamma,F)$ is the quantity $\delta_{\Gamma,F} \in \mathbb{R}$ defined by:
$$ \delta_{\Gamma,F} := \underset{n \rightarrow \infty}{\limsup} \  \underset{n-c < d(x,\gamma y) \leq n}{\frac{1}{n} \ln \ {\sum_{\gamma \in \Gamma}} \ e^{\int_x^{\gamma y} \tilde{F} }} ,$$
for any $x,y \in \mathbb{H}^d$ and any $c>0$. The critical exponent dosen't depend on the choice of $x,y$ and $c$.
\end{definition}

\begin{theorem}[\cite{PPS15}, section 3.6 and section 5.3]

Let $\Gamma \subset Iso^+(\mathbb{H}^d)$ be convex-cocompact, and note $M := \mathbb{H}^d / \Gamma$.
Let $F : T^1 M \rightarrow \mathbb{R}$ be a Hölder regular potential. Then there exists a unique (up to a scalar multiple) family of finite nonzero measures $(\mu_x)_{x \in \mathbb{H}^d}$ on $\partial_\infty \mathbb{H}^d$ such that, for all $\gamma \in \Gamma$, for all $x,y \in \mathbb{H}^d$ and for all $\xi \in \partial_\infty \mathbb{H}^d$:
\begin{itemize}
\item $\gamma_* \mu_x = \mu_{\gamma x}$
\item $d \mu_x(\xi) = e^{-C_{F-\delta_{\Gamma,F},\xi}(x,y)} d\mu_y(\xi)$
\end{itemize}
Moreover, these measures are all supported on the limit set $\Lambda_\Gamma$. We call them $(\Gamma,F)$-Patterson Sullivan densities.
\end{theorem}

\begin{remark}
Notice that Patterson-Sullivan densities only depend on the normalized potential $F-\delta_{\Gamma,F}$. Since $\delta_{\Gamma,F + \kappa} = \delta_{\Gamma,F} + \kappa $, replacing $F$ by $F-\delta_{\Gamma,F}$ allows us to work without loss of generality with potential satisfying $\delta_{\Gamma,F}=0$. We call such potential \emph{normalized}.
\end{remark}

The next estimate tells us, in a sense, that we can think of $\mu_0$ as a measure of a \say{fractal solid angle}, pondered by the potential. This is better understood by recalling that since the area of a hyperbolic sphere of large radius $r$ is a power of $\sim e^r$, then the solid angle of an object of diameter $1$ lying in that sphere is a power of $\sim e^{-r}$. In the following \say{Shadow lemma}, the object is a ball $B(y,R)$, at distance $d(x,y)$ from an observer at $x$.

\begin{proposition}[Shadow Lemma, \cite{PPS15} Lemma 3.10]
Let $R>0$ be large enough. There exists $C>0$ such that, for all $x,y \in \text{Hull}(\Lambda_\Gamma)$:

$$ C^{-1} e^{\int_x^{y} (\tilde{F} - \delta_{\Gamma,F})} \leq \mu_x\left( \mathcal{O}_x B(y,R) \right) \leq C e^{\int_x^{y} (\tilde{F}-\delta_{\Gamma,F})} .$$
\end{proposition}

\begin{definition}
The shadow lemma calls for the following hypothesis: we say that the potential $F$ satisfy the regularity assumptions (R) if $F$ is Hölder regular and if $\sup_{\Omega \Gamma} F < \delta_{\Gamma,F}$. 
\end{definition}

 \begin{remark}
By Lemma 3.3 in \cite{PPS15}, we see that we can construct potentials satisfying (R) as follow: choose some potential $F_0$ satisfying (R) (for example, the constant potential) and then choose any Holder map $E : T^1 M \rightarrow \mathbb{R}$ satisfying $2 \sup_{\Omega \Gamma} |E| < \delta_{\Gamma,F} - \sup_{\Omega {\Gamma}} F$. Then $F := F_0 + E$ satisfies the assumption (R). A similar assumption is introduced in \cite{GS14}.
 \end{remark}

The point of the assumption (R) is to ensure that the Patterson-Sullivan densities exhibit some regularity. This is possible because we have a tight control over the geometry of shadows.

\begin{lemma}
Let $\Gamma \subset Iso^+(\mathbb{H}^d)$ be convex-cocompact, and let $F : T^1 ({\mathbb{H}^d/\Gamma}) \rightarrow \mathbb{R}$ satisfy the regularity assumptions (R). Let $\delta_{reg} \in (0,1)$ such that $\delta_{reg} < \delta_{\Gamma,F} - \sup_{\Omega \Gamma} F$. Let $\mu$ denote some $(\Gamma,F)$-Patterson-Sullivan density. Then
$$ \exists C>0, \ \forall \xi \in \partial_\infty \mathbb{H}^{d}, \ \forall r>0, \  \mu(B(\xi,r)) \leq C r^{\delta_{reg}} ,$$
where the ball is in the sense of some visual distance.
\end{lemma}

\begin{proof}

First of all, since for all $p,q \in \mathbb{H}^d$, $\xi \in \partial_\infty \mathbb{H}^d \mapsto e^{C_{\xi}(p,q)}$ is continuous and since the ideal boundary is compact, we can easily reduce our statement to the case where $\mu$ is a Patterson-Sullivan density based on the center of the ball $o$. Moreover, since all the visual distances are equivalent, one can suppose that we are working for the visual distance based at $o$ too. Finally, since the support of $\mu_o$ is $\Lambda_\Gamma$, we can suppose without loss of generality that $B(\xi,r) \cap \Lambda_\Gamma \neq \emptyset$. Since in this case there exists some $\tilde{\xi} \in \Lambda_\Gamma$ such that $B(\xi,r) \subset B(\tilde{\xi},2r)$, we may further suppose without loss of generality that $\xi \in \Lambda_\Gamma$. \\

Now fix $\xi \in \Lambda_\Gamma$ and $r>0$. Let $x \in \text{Hull}(\Lambda_\Gamma)$ lay in the ray starting from $o$ and ending at $\xi$. Let $\rho \in [0,1]$, let $y \in \mathbb{H}^d$ such that $[o,y]$ is tangent to the sphere $S(x,\rho)$, and note $\eta$ the ending of the ray starting from $o$ and going through $y$. The hyperbolic law of sine (see Lemma A.4 and A.5 in \cite{LNP19}) allows us to compute directly:
$$ d_o(\xi,\eta) = \frac{1}{2} \cdot \frac{\sinh(\rho)}{\sinh(d(o,x))} .$$
It follows that there exists $C>0$ such that for all $\xi \in \Lambda_\Gamma$, for all $r>0$, there exists $x \in \text{Hull}(\Lambda_\Gamma)$ such that $e^{-d(o,x)} \leq C r$ and $B_o(\xi,r) \subset \mathcal{O}_o B(x,1) $. 
The desired bound follows from the shadow lemma, since the geodesic segment joining $o$ and $x$ lays in $\text{Hull}(\Lambda_\Gamma)$. \end{proof}

The regularity of the Patterson-Sullivan densities is going to allow us to state a second version of the contraction lemma. First, let us introduce a bit of notations. We fix, for all the duration of the paper, a large enough constant $C_{\Gamma}>0$. For $\gamma \in \Gamma$, we define $\kappa(\gamma) := d(o,\gamma o)$, $r_\gamma := e^{-\kappa(\gamma)}$ and $B_\gamma := \mathcal{O}_o B(\gamma o,C_\Gamma)$. By the hyperbolic law of sine, the radius of $B_\gamma$ is $\sinh(C_\Gamma)/\sinh(r_\gamma)$ (Lemma A.5 in \cite{LNP19}.) If $C_\Gamma$ is chosen large enough and when $\kappa(\gamma)$ is large, we get a radius of $\sim e^{C_\Gamma} r_\gamma \geq r_\gamma$. We have the following covering result.
\begin{lemma}[\cite{LNP19}, Lemma A.8]
Define $r_n := e^{-4 C_\Gamma n}$, and let $ S_n := \{ \gamma \in \Gamma \ , e^{-2 C_\Gamma} r_n \leq r_\gamma < r_n \}$. For all $n \geq 1$, the family $\{ B_\gamma \}_{\gamma \in S_n}$ cover $\Lambda_\Gamma$. Moreover, there exists $C>0$ such that:
$$ \forall n, \forall \xi \in \Lambda_\Gamma, \ \#\{ \gamma \in S_n \ , \ \xi \in B_\gamma \} \leq C .$$
\end{lemma}

Now, we are ready to state our second contraction lemma. Since the potential is not supposed bounded, a lot of technical bounds will only be achieved by working on the limit set or on its convex hull. One of the main goal of the second contraction lemma is then to replace $x_\gamma^m$ by a point $\eta_\gamma$ lying in the limit set.

\begin{lemma}[Second contraction lemma]
Let $\Gamma \subset \text{Iso}^+(\mathbb{H}^d)$ be convex-cocompact. Let $F : T^1( \mathbb{H}^d/\Gamma ) \rightarrow \mathbb{R}$ be a potential satisfying (R). Denote by $\mu_o \in \mathcal{P}(\Lambda_\Gamma)$ the associated Patterson-Sullivan density at $o$. Then there exists a family of points $(\eta_\gamma)_{\gamma \in \Gamma}$ such that, for any $\gamma \in \Gamma$ with large enough $\kappa(\gamma)$, we have $\eta_\gamma \in \Lambda_\Gamma \cap B_\gamma$, and moreover:

\begin{enumerate}
    \item there exists $c>0$ independant of $\gamma$ such that $d_o(\xi,\eta_\gamma) \geq r_\gamma/2 \Rightarrow d_o(\gamma^{-1} \xi, \gamma^{-1} \eta_\gamma) \geq c$,
    \item for all $\varepsilon_0 \in (0,\delta_{reg})$, there exists $C$ independant of $\gamma$ such that:
    $$ \int_{\Lambda_\Gamma} d_o(\gamma(\xi),\eta_\gamma)^{\varepsilon_0} d\mu_o(\xi) \leq C r_\gamma^{ \varepsilon_0 } .$$
\end{enumerate}

\end{lemma}

\begin{proof}
Recalling that the visual distance and the euclidean distance are equivalent on the unit sphere, if we forget about $\eta_\gamma$ and replace it by $x_\gamma^m$ instead, then the first point is a direct corollary of the first contraction lemma. We just have to check two points: first, since $\Gamma$ is discrete without torsion, there exists $c_0 >0$ such that for all $\gamma \in \Gamma \setminus \{Id\}$, $d(o,\gamma o) > c_0$. The second point is to check that the orders of magnitude of $r_\gamma$ and $\varepsilon_\gamma$ (quantity introduced in the first contraction lemma) are compatible. This can be checked using an explicit formula relating the hyperbolic distance with the euclidean one in the ball model: $r_\gamma = e^{-\kappa(\gamma)} = \frac{1-|\gamma(o)|}{1+|\gamma(o)|} \sim \varepsilon_\gamma/2$ (see \cite{Ra06}, exercise 4.5.1). We even have a large security gap for the first statement to hold (recall that the critical scale is $\sim \varepsilon_\gamma^2$). \\

We will use the strong contraction properties of $\Gamma$ to construct a point $\eta_\gamma \in \Lambda_\Gamma$ very close to $x_\gamma^m$. Since $\Gamma$ is convex-cocompact, we know in particular that $\text{diam}(\Lambda_\Gamma)>0$. Now let $\gamma \in \Gamma$ such that $\kappa(\gamma)$ is large enough. The first contraction lemma says that there exists $A_\gamma \subset \partial_\infty \mathbb{H}^d$ with $\text{diam} (A_\gamma) \leq C r_\gamma$ such that $\text{diam}(\gamma(A_\gamma^c)) \leq r_\gamma/10$. It follows that we can find a point $\widehat{\eta}_\gamma \in \Lambda_\Gamma$ such that $d(A_\gamma, \widehat{\eta}_\gamma)> \text{diam}(\Lambda_\Gamma)/3$. Fixing $\eta_\gamma := \gamma(\widehat{\eta}_\gamma)$ gives us a point satisfying $\eta_\gamma \in \Lambda_\Gamma$ and $d_o(\eta_\gamma,x_\gamma^m) \lesssim r_\gamma^{2}$. Hence $\eta_\gamma \in B_\gamma$, and moreover, any point $\xi$ satisfying $d_o(\xi,\eta_\gamma) \geq r_\gamma/2$ will satisfy $\gamma^{-1}(\xi) \in A_\gamma$, and so $d_o(\gamma^{-1}(\xi),\gamma^{-1}(\eta_\gamma)) \geq \text{diam}(\Lambda_\Gamma)/3$. This proves the first point. \\

For the second point: since the set $A_\gamma \subset \partial_\infty \mathbb{H}^{d}$ is of diameter $\leq C r_\gamma$ and satisfy that, for all $\xi \notin A_\gamma$, $d_o(\gamma(\xi),x_\gamma^m) \leq r_\gamma/10$, the upper regularity of $\mu_0$ yields
$$ \int_{\Lambda_{\Gamma}} d_o(\gamma(\xi),x_\gamma^m)^{\varepsilon_0} d\mu_o(\xi) \leq C \mu_o(A_\gamma) + \int_{A_\gamma^c} C r_\gamma^{\varepsilon_0} d\mu_o \leq C r_\gamma^{\varepsilon_0} .$$
The desired bound follows from $d_o(x_\gamma^m,\eta_\gamma) \lesssim r_\gamma$, using the triangle inequality.\end{proof}

\section{Patterson-Sullivan densities are stationary measures}

\subsection{Stationary measures}

In this subsection we define stationary measures and state our main theorem. 
\begin{definition}
Let $\nu \in \mathcal{P}(\Gamma)$ be a probability measure on $\Gamma \subset SO(n,1)$. Let $\mu \in \mathcal{P}(\partial_\infty \mathbb{H}^d)$. We say that $\mu$ is $\nu$-stationary if:
$$ \mu = \nu * \mu := \int_\Gamma \gamma_* \mu \ d\nu(\gamma) .$$
Moreover, we say that the measure $\nu$ has exponential moment if there exists $\varepsilon>0$ such that $ \int_\Gamma \|\gamma\|^{\varepsilon} d\nu(\gamma) < \infty$. Finally, we denote by $\Gamma_\nu$ the subgroup of $\Gamma$ generated by the support of $\nu$.
\end{definition}

\begin{theorem}
Let $\Gamma \subset \text{Iso}^+(\mathbb{H}^d)$ be a convex-cocompact group, and let $F : T^1(\mathbb{H}^d/\Gamma) \rightarrow \mathbb{R}$ be a potential on the unit tangent bundle satisfying (R). Let $x \in \mathbb{H}^d$ and let $\mu_x \in \mathcal{P}(\Lambda_\Gamma)$ denotes the $(\Gamma,F)$ Patterson-Sullivan density from $x$. Then there exists $\nu \in \mathcal{P}(\Gamma)$ with exponential moment (seen as a random walk in $SO(d,1)$) such that $\mu$ is $\nu$-stationary and such that $\Gamma_\nu = \Gamma$.
\end{theorem}

\begin{remark}
This result for $d=2$ was announced without proof by Jialun Li in \cite{Li18} (see remark 1.9). A proof in the case of constant potentials is done in the appendix of \cite{LNP19}. Our strategy is inspired by this appendix. For more details on stationary measures, see the references therein.
\end{remark}

First of all, a direct computation allows us to see that if $(\mu_x)_{x \in \mathbb{H}^d}$ are $(\Gamma,F)$ Patterson-Sullivan densities, then, for any $\eta \in SO_0(d,1)$ , $ (\eta_* \mu_{\eta^{-1} x})_{x \in \mathbb{H}^d} $ are $(\eta \Gamma \eta^{-1}, \eta_* F)$ Patterson-Sullivan densities. This remark allows us to reduce our theorem to the case where the basepoint $x$ is the center of the ball $o$. Our goal is to find $\nu \in \mathcal{P}(\Gamma)$ such that $\nu * \mu_o = \mu_o$.  Assuming that $F$ is normalized, this can be rewritten as follows:
$$ d\mu_o(\xi) = \sum_\gamma \nu(\gamma) d(\gamma_* \mu_o)(\xi) =  \sum_\gamma \nu(\gamma) e^{C_{F,\xi}(o,\gamma o)} d\mu_o(\xi) .$$
Hence, $\mu_o$ is $\nu$-stationary if
$$ \sum_{\gamma \in \Gamma} \nu(\gamma) f_\gamma = 1 \ \text{ on } \Lambda_\Gamma,$$
where $$f_\gamma(\xi) := e^{C_{F,\xi}(o,\gamma o)}.$$ 

\begin{remark}
Our main goal is to find a way to decompose the constant function $1$ as a sum of $f_\gamma$. Here is the intuition behind our proof. \\

Define $r_\gamma^{-F} := e^{-\int_o^{\gamma o} \tilde{F}} = f_\gamma(x^m_\gamma) \simeq f_\gamma(\eta_\gamma)$. The first thing to notice is that $f_\gamma$ looks like an approximation of unity centered at $x^m_\gamma$. Renormalizing yields the intuitive statement $r_\gamma^{F} f_\gamma \sim \mathbb{1}_{B_\gamma}$. The idea is that this approximation gets better as $\kappa(\gamma)$ becomes large. Once this observation is done, there is a natural \say{n-th approximation} operator that can be defined. For some positive function $R$, one can write:
$$ R \simeq \sum_{\gamma \in S_n} R(\eta_\gamma) \mathbb{1}_{B_\gamma} \simeq \sum_{\gamma \in S_n} R(\eta_\gamma) r_\gamma^{F} f_\gamma =: P_n R.$$
Proving that the operator $P_n$ does a good enough job at approximating some functions $R$ is the content of the \say{approximation lemma} 3.8. In particular, it is proved that, under some assumptions on $R>0$, we have $c R \leq P_n R \leq C R$. \\

The conclusion of the proof is then easy. We fix a constant $\beta>0$ small enough so that $c R \leq  \beta P_n R \leq R$. Then, we define by induction $R_0 = 1$ and $0 < R_{n+1} := R_n - \beta P_{n+1} R_n \leq (1-c) R_n$. By induction, this gives $R_n \leq (1-c)^n$, and hence $1 = R_0 - \lim_n R_n = \sum_{k} (R_{k}-R_{k+1})$ is a decomposition of $1$ as a sum of $f_\gamma$.
\end{remark}

\subsection{The approximation operator}

First, we collect some results on $f_\gamma$ that will allows us to think of it as an approximation of unity around $x_m^\gamma$ with width $r_\gamma$. The first point studies $f_\gamma$ near $x_\gamma^m$, the second point study the decay of $f_\gamma$ away from it, and the last point is a regularity estimate at the scale $r_\gamma$. To quantify this decay, we recall the notion of \emph{potential gap} (or gap map).

\begin{definition}
The following \say{potential gap} $D_F : \mathbb{H}^d \times \partial_\infty \mathbb{H}^d \times \partial_\infty \mathbb{H}^d$ is well defined and continuous:
$$D_{F,x}(\eta,\xi) := \exp \frac{1}{2} \lim_{t \rightarrow \infty} \left( \int_x^{\xi_t} \tilde{F} + \int_{\eta_t}^x \tilde{F} - \int_{\eta_t}^{\xi_t} \tilde{F} \right)$$
where $(\eta_t)$ and $(\xi_t)$ denotes any unit speed geodesic ray converging to $\eta$ and $\xi$.
\end{definition}

Under our assumptions, the gap map (for some fixed $x$) behaves like a distance on $\Lambda_\Gamma$: see \cite{PPS15}, section 3.4 for details. Finally, we denote by $\iota(v) = -v$ the \emph{flip map} on the unit tangent bundle. 

\begin{lemma}[Properties of $f_\gamma$]
There exists $C_{\text{reg}} \geq 2$ (independant of $C_\Gamma$) and $C_0 \geq 1$ (depending on $C_\Gamma)$ such that, for all $\gamma \in \Gamma$:
\begin{enumerate}
\item For all $\xi \in B_\gamma$, $$ r_\gamma^{F} f_\gamma(\xi) \in [e^{-C_0},e^{C_0}] .$$
 \item For all $\xi \in \Lambda_\Gamma$ such that $d_o(\xi,\eta_\gamma) \geq r_\gamma/2$,
$$ C_0^{-1}  D_o(\xi,\eta_\gamma)^{-2}  \leq r_\gamma^{-F \circ \iota} f_\gamma(\xi) \leq C_0 D_o(\xi,\eta_\gamma)^{-2} .$$
 \item For all $\xi,\eta \in \Lambda_\Gamma$ such that $d_o(\xi,\eta) \leq r_\gamma/e^2$,
 $$ \left| f_\gamma(\xi)/f_\gamma(\eta)-1 \right| \leq C_{\text{reg}} \cdot d_o(\xi,\eta)^\alpha r_\gamma^{-\alpha}. $$
 
\end{enumerate}

\end{lemma}

\begin{proof}
Recall that $f_\gamma(x_\gamma^m) = r_\gamma^{-F}$: the first point is then a consequence of Proposition 2.8. The third point also directly follows without difficulty. For the second item, recall that $\eta_\gamma \in B_\gamma$, so that by the same argument we get $$ r_\gamma^{-F \circ \iota} \simeq e^{C_{F \circ \iota,\eta_{\gamma}}(o,\gamma o)}. $$
Then, a direct computation yields:
$$ f_\gamma(\xi) r_\gamma^{-F \circ \iota} \simeq \exp \lim_{t \rightarrow \infty} \left( \left( \int_{\gamma o}^{\xi_t} \tilde{F} - \int_o^{\xi_t} \tilde{F}\right) - \left( \int_{(\eta_\gamma)_t}^{o} \tilde{F} - \int_{(\eta_\gamma)_t}^{\gamma o} \tilde{F} \right) \right) $$
$$ = \lim_{t \rightarrow \infty} \exp {\left( \left( - \int_o^{\xi_t} \tilde{F} - \int_{(\eta_\gamma)_t}^o \tilde{F}  + \int_{(\eta_\gamma)_t}^{\xi_t} \right) + \left( \int_{\gamma o}^{\xi_t} \tilde{F} + \int_{(\eta_\gamma)_t}^{\gamma o} \tilde{F} - \int_{(\eta_\gamma)_t}^{\xi_t} \tilde{F} \right) \right)} $$
$$ = \frac{D_{F,\gamma o}(\eta_\gamma,\xi)^2}{D_{F,o}(\eta_\gamma,\xi)^2} .$$
Under our regularity hypothesis (R), and because $\Gamma$ is convex6cocompact and $\xi,\eta \in \Lambda_\Gamma$, it is known that there exists $c_0>0$ such that $d_{\gamma o}(\xi,\eta_\gamma)^{c_0} \leq D_{F,\gamma o}(\eta_\gamma,\xi) \leq 1$ (see \cite{PPS15}, page 56). The second contraction lemma allows us to conclude, since $d_{\gamma o}(\xi,\eta_\gamma) = d_o(\gamma^{-1} \xi,\gamma^{-1}(\eta_\gamma)) \geq c$ under the hypothesis $d_o(\xi,\eta_\gamma) \geq r_\gamma/2$.
\end{proof}
To get further control over the decay rate of $f_\gamma$ away from $\eta_\gamma$, the following \say{role reversal} result will be helpful.
\begin{lemma}[symmetry]
Let $n \in \mathbb{N}$. Since $S_n$ is a covering of $\Lambda_\Gamma$, and by choosing $C_\Gamma$ larger if necessary, we know that for every $\eta \in \Lambda_\Gamma$ there exists $\tilde{\gamma}_\eta \in S_n$ such that $\eta \in S_{\tilde{\gamma}_\eta}$ and $d(\eta,\eta_{\tilde{\gamma}_\eta}) \leq C_{reg}^{-2/\alpha} r_\gamma$. Suppose that $d_o(\eta,\eta_\gamma) \geq r_\gamma$. Then:
$$ C^{-1} f_{\tilde{\gamma}_\eta}(\eta_\gamma)  \leq f_\gamma(\eta) \leq C f_{\tilde{\gamma}_\eta}(\eta_\gamma) $$
for some constant $C$ independant of $n$, $\gamma$ and $\eta$.
\end{lemma}

\begin{proof}
First of all, by the third point of the previous lemma, and since $d(\eta,\eta_{\tilde{\gamma}_\eta}) \leq C_{\text{reg}}^{-2/\alpha} r_\gamma$, we can write $ f_{\gamma}(\eta)/f_{\gamma}(\eta_{\tilde{\gamma}_\eta}) \in [1/2,3/2] $.
Then, by the previous lemma again, we see that
$$ \frac{f_\gamma(\eta)}{f_{\tilde{\gamma}_\eta}(\eta_\gamma)} \simeq \frac{f_{\gamma}(\eta_{\tilde{\gamma}_\eta})}{f_{\tilde{\gamma}_\eta}(\eta_\gamma)} \simeq \frac{D_{F,o}(\eta_{\tilde{\gamma}_\eta},\eta_\gamma)^2}{D_{F,o}(\eta_{\gamma},\eta_{\tilde{\gamma}_\eta})^{2}} \simeq 1,$$
where we used the quasi symmetry of the gap map (\cite{PPS15}, page 47).
\end{proof}

We are ready to introduce the $n$-th approximation operator. For some positive function $R : \Lambda_\Gamma \rightarrow \mathbb{R}^*_+$, define, on $\partial_\infty \mathbb{H}^d$, the following positive function:
$$ P_n R(\eta) := \sum_{\gamma \in S_n} R(\eta_\gamma) r_\gamma^F f_\gamma(\eta) .$$ 
The function $P_n R$ has the regularity of $f_\gamma$ for $\gamma \in S_n$.

\begin{lemma}
Choosing $C_\Gamma$ larger if necessary, the following hold. Let $n \in \mathbb{N}$ and let $\xi,\eta \in \Lambda_{\Gamma}$ such that $d_o(\xi,\eta) \leq r_{n+1}$. Then:
$$ \left| \frac{P_nR(\xi)}{P_n R(\eta)} - 1 \right| \leq \frac{1}{2} d_o(\xi,\eta)^\alpha r_{n+1}^{-\alpha} .$$

\end{lemma}

\begin{proof}
The regularity estimates on $f_\gamma$ and the positivity of $R$ yields:
$$ \left| P_nR(\xi) - P_nR(\eta) \right| \leq \sum_{\gamma \in S_n} R(\eta_\gamma) r_\gamma^F |f_\gamma(\xi)-f_\gamma(\eta)| $$
$$ \leq C_{\text{reg}}  \sum_{\gamma \in S_n} R(\eta_\gamma) r_\gamma^F f_\gamma(\eta) d_0(\xi,\eta)^{\alpha} r_{\gamma}^{-\alpha} = P_nR(\eta) \cdot C_{\text{reg}} e^{-2 \alpha C_{\Gamma}}  d_0(\xi,\eta)^{\alpha} r_{n+1}^{-\alpha} .$$
The bound follows.
\end{proof}

Now is the time where we combine all of our preliminary lemma to prove our main technical lemma: the approximation operator does a good enough job at approximating on $\Lambda_\Gamma$. Some natural hypothesis on $R$ are required: the function to approximate has to be regular enough at scale $r_n$, and has to have mild global variations (so that the decay of $f_\gamma$ away from $x_\gamma^m$ is still usefull).

\begin{lemma}[Approximation lemma]

Let $\varepsilon_0 \in (0,\delta_{reg})$ and $C_0 \geq 1$. Let $n \in \mathbb{N}$, and let $R : \Lambda_\Gamma \rightarrow \mathbb{R}$ be a positive function satisfying:
\begin{enumerate}
    \item For $\xi,\eta \in \Lambda_\Gamma$, if $d_o(\xi,\eta) \leq r_{n+1}$, then
    $$ \left|\frac{R(\xi)}{R(\eta)} - 1\right| \leq \frac{1}{2} \left(\frac{d_o(\xi,\eta)}{r_{n+1}}\right)^{\alpha} .$$
    \item For $\xi,\eta \in \Lambda_\Gamma$, if $d_o(\xi,\eta) > r_{n+1}$, then
    $$ R(\xi)/R(\eta) \leq C_0 d_o(\xi,\eta)^{\varepsilon_0} r_n^{-\varepsilon_0} .$$
\end{enumerate}
Then there exists $A \geq 1$ that only depends on $\varepsilon_0$ and $C_0$ such that, for all $\eta \in \Lambda_\Gamma$:
$$ A^{-1} R(\eta) \leq P_{n+1} R(\eta) \leq A R(\eta) .$$
\end{lemma}

\begin{proof}
Let $\eta \in \Lambda_\Gamma$. We have:
$$ P_{n+1} R(\eta) = \sum_{\gamma \in S_{n+1}} R(\eta_\gamma) r_\gamma^F f_{\gamma}(\eta) = \underset{\eta \in B_\gamma}{\sum_{\gamma \in S_{n+1}}} R(\eta_\gamma) r_\gamma^F f_{\gamma}(\eta) + \underset{\eta \notin B_\gamma}{\sum_{\gamma \in S_{n+1}}} R(\eta_\gamma) r_\gamma^F f_{\gamma}(\eta) .$$
The first sum is easily controlled: if $\eta \in B_\gamma$, then $R(\eta_\gamma) \simeq R(\eta)$ and $r_\gamma^F f_\gamma(\eta) \simeq 1$. Since $\eta$ is in a (positive and) bounded number of $B_\gamma$, we find
$$C^{-1} R(\eta) \leq \underset{\eta \in B_\gamma}{\sum_{\gamma \in S_{n+1}}} R(\eta_\gamma) r_\gamma^F f_{\gamma}(\eta) \leq C R(\eta),$$
which gives the lower bound since $R$, $r_\gamma^F$ and $f_\gamma$ are positive. To conclude, we need to get an upper bound on the residual term. Using $\text{diam}(B_\gamma) \lesssim r_n$, the symmetry lemma on $f_\gamma(\eta)$, the regularity and mild variations of $R$, and using the shadow lemma $r_\gamma^F \simeq \mu_o(B_\gamma)$, we get:
$$ \underset{\eta \notin B_\gamma}{\sum_{\gamma \in S_n}} R(\eta_\gamma) r_\gamma^F f_{\gamma}(\eta) \lesssim R(\eta) r_n^{-\varepsilon_0} \sum_{\eta \notin B_\gamma} r_\gamma^F d_o(\eta_\gamma,\eta)^{\varepsilon_0} f_{\tilde{\gamma}_\eta}( \eta_\gamma ) $$ $$ \lesssim R(\eta) \left( 1 + r_n^{-\varepsilon_0} \int_{B(\eta,r_{n+1})^c} d_o(\xi,\eta)^{\varepsilon_0} \ f_{\tilde{\gamma}_\eta}(\xi) d\mu_o(\xi) \right).$$
Finally, the second contraction lemma and the bound $d_o(\eta,\eta_{\tilde{\gamma}_\eta}) \lesssim r_n$ yields:
$$ \int_{B(\eta,r_{n+1})^c} d_o(\xi,\eta)^{\varepsilon_0} \ f_{\tilde{\gamma}_\eta}(\xi) d\mu_o(\xi) \leq \int_{\Lambda_\Gamma} d_o(\xi,\eta)^{\varepsilon_0} 
 \ f_{\tilde{\gamma}_\eta}(\xi) d\mu_o(\xi)  $$ $$= \int_{\Lambda_{\Gamma}} d_o(\xi,\eta)^{\varepsilon_0} d\mu_{\tilde{\gamma}_\eta o}(\xi) =  \int_{\Lambda_\Gamma} d_o(\tilde{\gamma}_\eta(\xi),\eta)^{\varepsilon_0} d\mu_{o}(\xi)$$ 
$$ \lesssim  r_n^{\varepsilon_0} + \int_{\Lambda_\Gamma} d_o(\tilde{\gamma}_\eta(\xi),\eta_{\tilde{\gamma}_\eta})^{\varepsilon_0} d\mu_{o}(\xi)  \lesssim r_n^{ \varepsilon_0} ,$$
which concludes the proof.
\end{proof}

\subsection{The construction of $\nu$}

In this last subsection, we construct the measure $\nu$ and conclude that $(\Gamma,F)$ Patterson-Sullivan densities are stationary measures for a random walk on $\Gamma$ with exponential moment. We can conclude by following the end of \cite{LNP19} very closely, but we will recall the last arguments for the reader's convenience. \\

Recall that the large constant $C_\Gamma \geq 1$ was fixed just before Lemma 2.16, and that $r_n := e^{-4 C_\Gamma n}$. Recall also that $\alpha>0$ is fixed by Lemma 2.8. We fix $\beta \in (0,1)$ small enough so that $1-\beta \geq e^{-4 C_\Gamma \alpha} + \beta$, and we choose $\varepsilon_0$ so that $r_n^{\varepsilon_0} = (1-\beta)^n$. By taking $\beta$ even smaller, we can suppose that $\varepsilon_0 < \delta_{reg}$. For this choice of $\varepsilon_0$, and for $C_0(\varepsilon_0) := 2 (1-\beta)^{-2} e^{4 C_\Gamma \varepsilon_0}$, the approximation lemma gives us a constant $A > 1$ such that, under the hypothesis of Lemma 3.9:
$$ \frac{\beta}{A^2} R \leq \frac{\beta}{A} P_{n+1} R \leq \beta R .$$
We then use $P_n$ to successively take away some parts of $R$. Define, by induction, $R_0 := 1$ and
$$ R_{n+1} := R_{n} - \frac{\beta}{A} P_{n+1} R_n \leq R_n.$$
For the process to work as intended, we need to check that $R_n$ satisfies the hypothesis of the approximation lemma.

\begin{lemma}
Let $n \in \mathbb{N}$. The function $R_n$ is positive on $\Lambda_{\Gamma}$, and for any $\xi,\eta \in \Lambda_{\Gamma}$:
\begin{enumerate}
\item If $d_o(\xi,\eta) \leq r_{n+1}$ then $$ \left| \frac{R_n(\xi)}{R_n(\eta)} - 1 \right| \leq \frac{1}{2} d_o(\xi,\eta)^{\alpha} r_{n+1}^{-\alpha} $$
\item If $d_o(\xi,\eta) > r_{n+1}$, then $$ R_n(\xi)/R_n(\eta) \leq C_0(\varepsilon_0) \cdot  d_o(\xi,\eta)^{\varepsilon_0} (1-\beta)^{-n} .$$
\end{enumerate}
\end{lemma}

\begin{proof}

The proof goes by induction on $n$. The case $n=0$ is easy: the first point holds trivially and the second holds since $C_0(\varepsilon_0) r_{1}^{\varepsilon_0} = C_0(\varepsilon_0) e^{- 4 C_\Gamma \varepsilon_0} \geq 1$. Now, suppose that the result hold for some $n$. In this case, the approximation lemma yields
$$ R_{n+1} = R_n - \frac{\beta}{A} P_n R \geq (1-\beta) R_n, $$
and in particular $R_{n+1}$ is positive. Let us prove the first point: consider $\xi,\eta \in \Lambda_{\Gamma}$ such that $d_o(\xi,\eta) \leq r_{n+2}$. Then Lemma 3.8 gives
$$ \left| \frac{\beta}{A} P_{n+1} R_n(\xi) - \frac{\beta}{A} P_{n+1} R_n(\eta) \right| \leq \frac{1}{2} \left(\frac{\beta}{A} P_{n+1} R_n(\eta)\right) \cdot d_o(\xi,\eta)^{\alpha} r_{n+2}^{-\alpha} \leq \frac{1}{2} \beta R_n(\eta) \cdot d_o(\xi,\eta)^{\alpha} r_{n+2}^{-\alpha}. $$
Hence, using the induction hypothesis:
$$ \left| R_{n+1}(\xi) - R_{n+1}(\eta) \right| \leq \left| R_{n}(\xi) - R_{n}(\eta) \right| + \left| \frac{\beta}{A} P_{n+1} R_n(\xi) - \frac{\beta}{A} P_{n+1} R_n(\eta) \right| $$ $$ \leq \frac{1}{2} \left( r_{n+1}^{-\alpha} + \beta r_{n+2}^{-\alpha} \right) R_{n}(\eta) d_o(\xi,\eta)^{\alpha}  $$
$$ \leq \frac{1}{2} \frac{e^{-4 C_\Gamma \alpha} + \beta}{1-\beta} R_{n+1}(\eta) d_o(\xi,\eta)^{\alpha} r_{n+2}^{-\alpha} .$$
Recalling the definition of $\beta$ gives the desired bound. It remains to prove the second point. First of all, notice that, for any $\xi$ and $\eta$, we have:
$$ \frac{R_{n+1}(\xi)}{R_{n+1}(\eta)} \leq (1-\beta)^{-1} \frac{R_n(\xi)}{R_n(\eta)} .$$
Now, suppose that $d_o(\xi,\eta) \in (r_{n+2},r_{n+1}]$. The induction hypothesis gives $R_n(\xi)/R_n(\eta) \leq 1+|R_n(\xi)/R_n(\eta)-1| \leq 2$, and so:
$$ \frac{R_{n+1}(\xi)}{R_{n+1}(\eta)} \leq \frac{2}{1-\beta} =  \frac{2}{1-\beta} r_{n+2}^{\varepsilon_0} (1-\beta)^{-(n+2)} \leq \frac{2}{(1-\beta)^2} \cdot d_o(\xi,\eta)^{\varepsilon_0} (1-\beta)^{-{n+1}} ,$$
which proves the bound. Finally, suppose that $d_o(\xi,\eta) > r_{n+1}$. In this case, the induction hypothesis directly yields
$$ \frac{R_{n+1}(\xi)}{R_{n+1}(\eta)} \leq \frac{1}{1-\beta} \frac{R_n(\xi)}{R_n(\eta)} \leq C_0(\varepsilon_0) d_o(\xi,\eta)^{\varepsilon_0} (1-\beta)^{-(n+1)} ,$$
and the proof is done. \end{proof}

We are ready to prove Theorem 3.2, following Li in \cite{LNP19}.

\begin{proof}

The previous lemma ensure that for all $n$, the function $R_n$ satisfies the hypothesis of the approximation lemma. Hence, we can write for all $n$

$$ R_{n+1} = R_n - \frac{\beta}{A} P_{n+1} R_n \leq \left(1-\frac{\beta}{A^2}\right) R_n, $$
so that by induction:
$$ R_n \leq \left(1- \frac{\beta}{A^2} \right)^n \longrightarrow 0.$$
It follows that $$ 1 = R_0 - \lim_{n} R_n = \sum_{n=1}^\infty (R_{n-1} - R_{n} ) = \frac{\beta}{A} \sum_{n=1}^\infty P_{n} (R_{n-1}), $$
in other words:
$$ 1 = \sum_{n = 1}^\infty \sum_{\gamma \in S_n} \frac{\beta}{A} R_{n-1}(\eta_\gamma) r_\gamma^F \cdot f_\gamma .$$
Letting $$ \nu(\gamma) :=  \frac{\beta}{A} R_{n-1}(\eta_\gamma) r_\gamma^F \ \text{if} \ \gamma \in S_n, \quad \nu(\gamma):=0 \ \text{if} \ \gamma \neq \bigcup_k S_k  $$
 gives us a probability measure on $\Gamma$ (since $\int f_\gamma d\mu_o =1$) satisfying $\nu * \mu_0 = \mu_0$, by the remarks made section 3.1. Checking that the measure has exponential moment is easy since $\|\gamma\| \lesssim r_\gamma^{-1}$ by Remark 2.4. Hence, by the shadow lemma $r_\gamma^F \simeq \mu_o(B_\gamma)$ and since $S_n$ covers each point a bounded number of time, we get:
$$ \int_{\Gamma} \|\gamma\|^{\varepsilon} d\nu \lesssim \sum_{n} \sum_{\gamma \in S_n} \nu(\gamma) r_\gamma^{-\varepsilon}  $$
$$ \lesssim \sum_n \left( \sum_\gamma \mu_o(B_\gamma) \right) (1-\beta/A^2)^n e^{\varepsilon 4 C_\Gamma n} < \infty $$
if $\varepsilon$ is small enough. Finally, we show that group $\Gamma_\nu$ spanned by the support of $\nu$ is $\Gamma$. To see this, say that $C_\Gamma$ was chosen so large that $C_\Gamma \geq 6 \text{diam} (\text{Hull}(\Lambda_\Gamma)/\Gamma)$. In this case, there exists $\gamma_1 \in S_1$ such that $d(o,\gamma_1 o) \in [|\ln r_1| + C_\Gamma/2, |\ln r_1| + 3 C_\Gamma/2]$. Then, any $\gamma \in \Gamma$ such that $d(o,\gamma o) \leq C_\Gamma/2$ satisfies $\gamma_1 \gamma \in S_1$. In particular:
$$ \{ \gamma \in \Gamma \ , \ d(o,\gamma o) \leq C_\Gamma/2 \} \subset \Gamma_\nu,$$
and it is then well known (see for example Lemma A.14 in \cite{LNP19}) that this set spans the whole group $\Gamma$ as soon as $C_\Gamma/2$ is larger than 3 times the diameter of $\text{Hull}({\Lambda_\Gamma})/\Gamma$. 
\end{proof}

\section{Consequences on equilibrium states}

Now that we have proved Theorem 1.2, Corollary 1.3 follows directly from \cite{Li20} (since, when $d=2$, being Zariski dense is equivalent to being non-elementary). To see how this statement induce some knowledge over equilibrium states, let us recall more precisely the link between the latter and Patterson-Sullivan densities. First, recall that the Hopf coordinates $$\text{Hopf} : \left((\partial_\infty \mathbb{H}^2 \times \partial_\infty \mathbb{H}^2)\setminus \mathcal{D}\right) \times \mathbb{R} \longrightarrow T^1 \mathbb{H}^2$$ allows us to smoothly identify the unit tangent bundle of $\mathbb{H}^2$ with a torus minus the diagonal times $\mathbb{R}$ by the following process. For any $v^+ \neq v^- \in \partial_\infty \mathbb{H}^2$, and for any $t \in \mathbb{R}$, $ \text{Hopf}(v^+,v^-,t) := v $ is the unique vector $v \in T^1 M$ lying on the geodesic $]v^-,v^+[$ such that $\tilde{p}(\phi_{-t}(v))$ is the closest point to $o$ on this geodesic. 
We will denote by $(\partial_{v^+},\partial_{v^-}, \partial_{t})$ the induced basis of $T(T^1 M)$ in these coordinates. Finally, recall that $\iota$ denotes the flip map.

\begin{theorem}[\cite{PPS15}, Theorem 6.1]

Let $\Gamma \subset Iso^+(\mathbb{H}^d)$ be convex cocompact, $M := \mathbb{H}^d/\Gamma$, and $F:T^1 M \rightarrow \mathbb{R}$ be a normalized and Hölder-regular potential. Denote by $m_F \in \mathcal{P}(T^1 M)$ the associated equilibrium state, and let $\tilde{m}_F$ be its $\Gamma$-invariant lift on $T^1 \mathbb{H}^d$. Denotes by $\mu_{x}^F$ the $(\Gamma,F)$ Patterson-Sullivan density with basepoint $x$. Then, for any choice of $x \in \mathbb{H}^d$, the following identity hold in the Hopf coordinates (up to a multiplicative constant $c_0 > 0$):
$$ c_0 \cdot d\tilde{m}_F(v^+,v^-,t) = \frac{d\mu_{x}^F(v^+) d\mu^{F \circ \iota}_{x}(v^-)  dt}{D_{F,x}(v^+,v^-)^{2}} .$$
\end{theorem}

We are now ready to prove Fourier decay for $m_F$. To do a clean proof, we write down three lemmas corresponding to Fourier decay in the three directions $(\partial_{v^+}, \partial_{v^-},\partial_t)$. We will then combine all of them to get the desired result.

\begin{lemma}
Under the conditions of Theorem 1.2, there exists $\varepsilon>0$ such that the following hold. Let $R \geq 1$ and let $\chi : T^1 \mathbb{H}^2 \rightarrow \mathbb{R}$ be a Hölder map supported on some compact $K$. There exists $C \geq 1$ such that for any $C^2$ function $\varphi : T^1 \mathbb{H}^2 \rightarrow \mathbb{R}$ satisfying $\|\varphi\|_{C^2} + (\inf_{K} |\partial_{v^+} \varphi|)^{-1} \leq R $, we have:
$$ \forall \xi \in \mathbb{R}^*, \ \left|\int_{T^1 \mathbb{H}^2} e^{i \xi \varphi(v)} \chi(v) d\tilde{m}_F(v) \right| \leq \frac{C}{ |\xi|^{\varepsilon}} .$$
\end{lemma}

\begin{proof}
Denotes $\tilde{\varphi}$ and $\tilde{\chi}$ the functions $\varphi,\chi$ seen in the Hopf coordinates. We get, for some large $a>0$ depending only on the support of $\chi$:
$$ c_0 \int_{T^1 \mathbb{H}^2} e^{i \xi \varphi(v)} \chi(v) d\tilde{m}_F(v)   = \int_{-a}^a \int_{\Lambda_\Gamma}  \left( \int_{\Lambda_\Gamma} e^{i \xi \tilde{\varphi}(v^+,v^-,t)} \frac{\tilde{\chi}(v^+,v^-,t)}{D_{F,o}(v^+,v^-)^2} d\mu_o^F(v^+) \right) d\mu_o^{F \circ \iota}(v^-) dt .$$
Now, since $\tilde{\chi}$ is supported in a compact subset of $\left((\partial_\infty \mathbb{H}^2 \times \partial_\infty \mathbb{H}^2)\setminus \mathcal{D}\right) \times \mathbb{R}$, and since $D_F$ is uniformly Hölder (and doesn't vanish) on a compact subset of $\Lambda_\gamma \times \Lambda_\Gamma \setminus \mathcal{D}$ (see \cite{PPS15}, Lemma 3.6 and Proposition 3.5), and finally since $\partial_{v^+} \tilde{\varphi} \neq 0$ on the compact support of $\tilde{\chi}$, we see that Corollary 1.3 applies to the inner integral. (Notice that we can always extend $D_F$ outside of $\Lambda_\Gamma \times \Lambda_\Gamma \setminus \mathcal{D}$ so that it becomes Holder on all $(\partial_\infty \mathbb{H}^2 \times \partial_\infty \mathbb{H}^2) \setminus \mathcal{D}$, see \cite{Mc34}.) This gives the desired bound.
\end{proof}

\begin{lemma}
Under the conditions of Theorem 1.2, there exists $\varepsilon>0$ such that the following hold. Let $R \geq 1$ and let $\chi : T^1 \mathbb{H}^2 \rightarrow \mathbb{R}$ be a Hölder map supported on some compact $K$. There exists $C \geq 1$ such that for any $C^2$ function $\varphi : T^1 \mathbb{H}^2 \rightarrow \mathbb{R}$ satisfying $\|\varphi\|_{C^2} + (\inf_{K} |\partial_{v^-} \varphi|)^{-1} \leq R $, we have:
$$ \forall \xi \in \mathbb{R}^*, \ \left|\int_{T^1 \mathbb{H}^2} e^{i \xi \varphi(v)} \chi(v) d\tilde{m}_F(v) \right| \leq \frac{C}{ |\xi|^{\varepsilon}} .$$
\end{lemma}

\begin{proof}
We need to check that when $F$ satisfies the regularity assumptions (R), then $F \circ \iota$ satisfies them too. This is easy, since $\sup_{\Omega \Gamma} F \circ \iota = \sup_{\Omega \Gamma} F < \delta_{\Gamma,F} = \delta_{\Gamma, F \circ \iota} $ by Lemma 3.3 in \cite{PPS15}. Moreover, $F \circ \iota$ is still Hölder regular. Hence, one can apply our previous lemma with $F$ replaced by $F \circ \iota$, and conclude.
\end{proof}

\begin{lemma}
Under the conditions of Theorem 1.2, let $R \geq 1$ and let $\chi : T^1 \mathbb{H}^2 \rightarrow \mathbb{R}$ be a $\alpha$-Hölder map supported on some compact $K$. There exists $C \geq 1$ such that, for any $C^2$ function $\varphi : T^1 \mathbb{H}^2 \rightarrow \mathbb{R}$ satisfying $\|\varphi\|_{C^2} + (\inf_K |\partial_{t} \varphi|)^{-1} \leq R$, we have:

$$ \forall \xi \in \mathbb{R}^*, \ \left|\int_{T^1 \mathbb{H}^2} e^{i \xi \varphi(v)} \chi(v) d\tilde{m}_F(v) \right| \leq \frac{C}{ |\xi|^{\alpha}} $$
\end{lemma}

\begin{proof}
The proof is classic. We have, for some compact $\tilde{K} \subset  \left(\partial_\infty \mathbb{H}^2 \times \partial_\infty \mathbb{H}^2 \right) \setminus \mathcal{D}$ and for some large enough $a>0$ depending only on the support of $\chi$:

$$ c_0 \int_{T^1 \mathbb{H}^2} e^{i \xi \varphi(v)} \chi(v) d\tilde{m}_F(v)   = \iint_{\tilde{K}}  \left( \int_{-a}^a e^{i \xi \tilde{\varphi}(v^+,v^-,t)} {\tilde{\chi}(v^+,v^-,t)} dt \right) D_{F,o}(v^+,v^-)^{-2} d(\mu_o^F \otimes \mu_o^{F \circ \iota})(v^+,v^-) .$$
We then work on the inner integral. When $\tilde{\chi}$ is $C^1$, we can conclude by an integration by parts. So a way to conclude is to approximate $\tilde{\chi}$ by a $C^1$ map. Fix some smooth bump function $\rho : \mathbb{R} \rightarrow \mathbb{R}^+$ such that $\rho$ is zero outside $[-2,2]$, one inside $[-1,1]$, increasing on $[-2,-1]$ and decreasing on $[1,2]$. For any $\varepsilon>0$, set $$\tilde{\chi}_\varepsilon(\cdot,\cdot,t) := \int_\mathbb{R}\tilde{\chi}(\cdot,\cdot,t-x) \rho(x/\varepsilon) dx/\varepsilon.$$
This function is smooth on the $t$-variable. Moreover, if we denote by $\alpha$ a Hölder exponent for $\chi$, then a direct computation yields:

$$ \|\tilde{\chi}_\varepsilon - \tilde{\chi}\|_\infty \lesssim \varepsilon^\alpha, \quad \|\partial_t \tilde{\chi}_\varepsilon \|_\infty \lesssim \varepsilon^{-(1-\alpha)} .$$
Hence:

$$ \left|\int_{-a}^a e^{i \xi \tilde{\varphi}} {\tilde{\chi}} dt\right| \leq 2 a \| \tilde{\chi} - \tilde{\chi}_\varepsilon \|_\infty + \left|\int_{-a}^a e^{i \xi \tilde{\varphi}} {\tilde{\chi}_\varepsilon} dt\right| . $$
To control the integral on the right, we do our aforementioned integration by parts:

$$ \int_{-a}^a e^{i \xi \tilde{\varphi}} {\tilde{\chi}_\varepsilon} dt = \int_{-a}^a \frac{i \xi \partial_t \tilde{\varphi}}{i \xi \partial_t \tilde{\varphi}} e^{i \xi \tilde{\varphi}} {\tilde{\chi}_\varepsilon} dt $$
$$ = \left[ \frac{\tilde{\chi}_\varepsilon}{i \xi \partial_t \tilde{\varphi}} e^{i \xi \tilde{\varphi}} \right]_{t=-a}^{t=a} - \frac{i}{\xi} \int_{-a}^a \partial_t\left( \frac{\tilde{\chi}_\varepsilon}{\partial_t\tilde{\varphi}} \right) e^{i \xi \tilde{\varphi}} dt ,$$
so that $$ \left| \int_{-a}^a e^{i \xi \tilde{\varphi}} {\tilde{\chi}_\varepsilon} dt \right| \lesssim |\xi|^{-1} {\varepsilon^{-(1-\alpha)}} .$$
Finally, choosing $\varepsilon = 1/|\xi|$ yields
$$ \left| \int_{T^1 \mathbb{H}^2} e^{i \xi \varphi(v)} \chi(v) d\tilde{m}_F(v) \right| \lesssim \varepsilon^\alpha + \varepsilon^{-(1-\alpha)} |\xi|^{-1} \lesssim |\xi|^{-\alpha} ,$$ which is the desired bound. \end{proof}

\begin{theorem}
Under the conditions of Theorem 1.2, there exists $\varepsilon>0$ such that the following holds. Let $R \geq 1$ and let $\chi : T^1 M \rightarrow \mathbb{R}$ be a Hölder map supported on some compact $K$. There exists $C \geq 1$ such that, for any $C^2$ function  $\varphi : T^1 M \rightarrow \mathbb{R}$ satisfying $\|\varphi\|_{C^2}+(\inf_K \|d\varphi\|)^{-1} \leq R$, we have:
$$ \forall \xi \in \mathbb{R}^*, \ \left|\int_{T^1 M} e^{i \xi \varphi} \chi d m_F \right| \leq \frac{C}{ |\xi|^{\varepsilon}} $$
\end{theorem}

\begin{proof} 
First of all, choose $\tilde{\chi}: T^1 \mathbb{H}^2 \rightarrow \mathbb{R}$ a lift of $\chi$ supported on a fundamental domain of $\Gamma$. Denote by $\tilde{K} \subset T^1 \mathbb{H}^2$ its (compact) support. Lift $\varphi$ to a $\Gamma$-invariant map $\tilde{\varphi} : T^1 \mathbb{H}^2 \rightarrow \mathbb{R}$. We then have:

$$ \int_{T^1 M} e^{i \xi \varphi} \chi d m_F = \int_{T^1 \mathbb{H}^2} e^{i \xi \tilde{\varphi}} \tilde{\chi} d \tilde{m}_F .$$

Now, consider the map $\mathcal{B}_\varphi : T^1 \mathbb{H}^2 \rightarrow \mathbb{R}^3$ defined by $\mathcal{B}_\varphi(v) := \left((d \tilde{\varphi})_v(\partial_{v^+}), (d \tilde{\varphi})_v(\partial_{v^-}), (d \tilde{\varphi})_v(\partial_{t}) \right)$.
Since $((\partial_{v^+})_v,(\partial_{v^-})_v,(\partial_{t})_v)$ is a basis of $T_v (T^{1} \mathbb{H}^2)$ for any $v \in T^1 \mathbb{H}^2$, and since $d\tilde{\varphi}$ doesn't vanish on $\tilde{K}$, we see that $\mathcal{B}_\varphi(\tilde{K}) \subset \mathbb{R}^3 \setminus \{0\}$. By uniform continuity of $\mathcal{B}_\varphi$ on the compact $\tilde{K}$, it follows that there exists $c_0>0$ such that we can cover $\tilde{K}$ by a finite union of compact balls $(B_j)_{j \in J}$ satisfying:
$$ \forall j \in J, \ \exists e \in \{v^+,v^-,t\}, \ \forall v \in B_j, \ |\partial_e \tilde{\varphi}(v)|>c_0 .$$
To conclude, we consider a partition of unity $(\widehat{\chi}_j)_j$ adapted to the cover $(B_j)_j$, and we write:

$$  \int_{T^1 \mathbb{H}^2} e^{i \xi \tilde{\varphi}} \tilde{\chi} d \tilde{m}_F  = \sum_{j \in J} \int_{B_j} e^{i \xi \tilde{\varphi}} \tilde{\chi} \widehat{\chi}_j d\tilde{m}_F. $$
Each of the inner integrals is then controlled by either Lemma 4.2, Lemma 4.3 or Lemma 4.4.
\end{proof}

\begin{remark}

We recover our main Theorem 1.4 as a particular case of Theorem 4.5. Indeed, if $\varphi : K \subset T^1 M \rightarrow \mathbb{R}^3$ is a $C^2$ local chart, then for any $\zeta \in \mathbb{R}^3 \setminus \{0\}$, one may write:

$$ \int_{T^1 M} e^{i \zeta \cdot \varphi(v)} \chi(v) d m_F(v) = \int_{T^1 M} e^{i |\zeta| (\zeta/|\zeta|) \cdot \varphi(v)} \chi(v) d m_F(v) \leq C |\zeta|^{- \varepsilon} ,$$
since the map $(u,v) \in \mathbb{S}^2 \times K \mapsto u \cdot (d\varphi)_v \in T_v^*(T^1 M)$ doesn't vanish (because the range of $(d\varphi)_v$ isn't contained in a plane). Notice that we used the uniformity of the constants $C \geq 1$ given by the phases $u \cdot (d\varphi)$.
\end{remark}

\begin{appendices}

\section{On the Fourier dimension}

\subsection{The upper and lower Fourier dimension}

We naturally want to make sense of the Fourier dimension of the non-wandering set of the geodesic flow, so that we can write a sentence of the form: \say{$\dim_F \text{NW}(\phi) > 0$}. But since $NW(\phi)$ is a subset of an abstract manifold, the usual definition doesn't apply. In this appendix, we suggest some definitions that one could choose to talk about the Fourier dimension of a compact set lying in an abstract manifold. \\

First of all, recall that the Fourier dimension of a probability measure $\mu \in \mathcal{P}(E)$, supported on some compact set $E \subset \mathbb{R}^d$, can be defined as:
$$ \dim_F(\mu) := \sup \{\alpha \geq 0 \ | \ \exists C \geq 1, \forall \xi \in \mathbb{R}^d \setminus \{0\}, \ |\widehat{\mu}(\xi)| \leq C |\xi|^{-\alpha/2} \}, $$
where the Fourier transform of $\mu$ is defined by $$ \widehat{\mu}(\xi) := \int_{E} e^{-2 i \pi \xi \cdot x} d\mu(x) .$$
The Fourier dimension of a compact set $E \subset \mathbb{R}^d$ is then defined as $$ \dim_F(E) := \sup \{ \min(d,\dim_F \mu) , \ \mu \in \mathcal{P}(E) \} \leq \dim_H E. $$

To define the Fourier dimension of a measure lying in a abstract manifold, a natural idea is to look at our measure into local charts. But this suppose that we have a meaningful way to \say{localize} the usual definition of the Fourier dimension. This is the content of the next well known lemma.

\begin{lemma}
Let $E \subset \mathbb{R}^d$ be a compact set. Let $\mu \in \mathcal{P}(E)$. Let $\varepsilon>0$. Denote by $\text{Bump}(\varepsilon)$ the set of smooth functions with support of diameter at most $\varepsilon$. Then:

$$ \dim_F \mu = \inf \{ \dim_F(\chi d\mu) \ | \ \chi \in \text{Bump}(\varepsilon) \} .$$
\end{lemma}

\begin{proof}
Let $E \subset \mathbb{R}^d$ be a fixed compact set, and let $\mu \in \mathcal{P}(E)$ be a fixed (borel) probability measure supported on $E$. First of all, consider a finite covering of the compact set $E$ by balls $(B_i)_{i \in I}$ of radius $\varepsilon$. Consider an associated partition of unity $(\chi_i)_{i \in I}$. Then, for all $\alpha < \inf_{\chi} \dim_F (\chi d\mu)$, there exists $C \geq 1$ such that:
$$ |\widehat{\mu}(\xi)| = \left|\sum_{i \in I} \widehat{\chi_i d\mu}(\xi)\right| \leq C |\xi|^{-\alpha} .$$
Hence $\dim_F \mu \geq \alpha$. Since this hold for any $\alpha < \inf \{ \dim_F(\chi d\mu) \ | \ \chi \in \text{Bump}(\varepsilon) \} $, this yields $\dim_F \mu \geq \inf \{ \dim_F(\chi d\mu) \ | \ \chi \in \text{Bump}(\varepsilon) \}.$ Now we prove the other inequality. \\

Fix some smooth function with compact support $\chi$. Its Fourier transform $\widehat{\chi}$ is in the Schwartz class: in particular, for all $N \geq d+1$, there exists $C_N$ such that $\widehat{\chi}(\eta) \leq C_N |\eta|^{-N}$ for all $\eta \in \mathbb{R}^d \setminus \{0\}$. Let $\alpha < \alpha' < \dim_F \mu$. Then there exists $C \geq 1$ such that $|\widehat{\mu}(\xi)| \leq C |\xi|^{-\alpha'}$ for all $\xi \in \mathbb{R}^d \setminus \{0\}$. Now, notice that:
$$ \widehat{\chi d\mu}(\xi) = \widehat{\chi} * \widehat{\mu}(\xi)  = \int_{\mathbb{R}^d} \widehat{\chi}(\eta) \widehat{\mu}(\xi-\eta) d\eta. $$
We cut the integral in two parts, depending on some radius $r>0$ that we choose to be $r := |\xi|^{1-\varepsilon}$, where $\varepsilon := 1-\alpha/\alpha'$. We suppose that $|\xi| \geq 2$. In this case, a direct computation show that whenever $\eta \in B(0,r)$, we have $|\xi|^{1-\varepsilon} \leq C |\xi-\eta|$. We are finally ready to conclude our computation:
$$ \left|  \widehat{\chi d\mu}(\xi)\right| \leq \left| \int_{B(0,r)} \widehat{\chi}(\eta) \widehat{\mu}(\xi-\eta) d\eta \right| +  \left| \int_{B(0,r)^c} \widehat{\chi}(\eta) \widehat{\mu}(\xi-\eta) d\eta \right| $$
$$   \lesssim_N  \int_{\mathbb{R}^d} |\widehat{\chi}(\eta)| d\eta \cdot \frac{C}{|\xi|^{(1-\varepsilon)\alpha'}} +   \int_{B(0,r)^c} \frac{1}{|\eta|^N} d\eta $$
$$ \lesssim_N \frac{1}{|\xi|^\alpha} + r^{N-d} \int_{B(0,1)^c} \frac{1}{|\zeta|^N} d\zeta \lesssim \frac{1}{|\xi|^{\alpha}} $$
if $N$ is choosen large enough. It follows that $\dim_F(\chi d\mu) \geq \alpha$, and this for any $\alpha<\dim_F \mu$, so $\dim_F(\chi d\mu) \geq \dim_F (\mu)$. Taking the infimum in $\chi$ yields the desired inequality.
\end{proof}

Now we understand how the Fourier dimension of a measure $\mu$ can be computed by looking at the local behavior of $\mu$. But another, much harder problem arise now: the Fourier dimension of a measure depends very much on the embedding of this measure in the ambiant space. In concrete terms, the Fourier dimension is not going to be independant on the choice of local charts. A way to introduce an "intrinsic" quantity related to the Fourier dimension of a measure would be to take the supremum or the infimum under all those charts.
We directly give our definition in the context of a manifold.

\begin{definition}
Let $M$ be a smooth manifold of dimension $d$. Let $E\subset M$ be a compact set. Let $\mu \in \mathcal{P}(E)$. Let $k \in \mathbb{N}^*$. Let $\text{Bump}(E)$ denote the set of all smooth functions $\chi : M \rightarrow \mathbb{R}$ such that $\text{supp}(\chi)$ is contained in a local chart. We denote by $\text{Chart}(\chi,C^k)$ the set of all $C^k$ local charts $\varphi : U \rightarrow \mathbb{R}^d$, where $U \supset \text{supp}(\chi)$ is an open set containing the support of $\chi$. Now, define the \underline{lower Fourier dimension} of $\mu$ by $C^k$ charts of $M$ by:
$$ \underline{\dim}_{F,C^k}(\mu) := \inf_{\chi \in \text{Bump}(E)} \inf \{ \dim_F( \varphi_*(\chi d\mu) ) , \ \varphi \in \text{Chart}(\chi,C^k) \}. $$
Similarly, define the \underline{upper Fourier dimension} of $\mu$ by $C^k$ charts of $M$ by:
$$ \overline{\dim}_{F,C^k}(\mu) := \inf_{\chi \in \text{Bump}(E)} \sup\{ \dim_F( \varphi_*(\chi d\mu) ) , \ \varphi \in \text{Chart}(\chi,C^k) \}. $$
\end{definition}

\begin{definition}
Let $M$ be a smooth manifold of dimension $d$. Let $E\subset M$ be a compact set. Let $\mu \in \mathcal{P}(E)$. We define the lower Fourier dimension of $\mu$ by:
$$ \underline{\dim}_F(\mu) = \underline{\dim}_{F,C^\infty}(\mu).$$
\end{definition}

\begin{remark}
The lower Fourier dimension test if, for any localization $\chi d\mu$ of $\mu$, and for any smooth local chart $\varphi$, one has some decay of the Fourier transform of $\varphi_*(\chi d \mu)$. We then take the infimum of all the best decay exponents. This quantity is \underline{$C^\infty$-intrinsic} in the following sense: if $\Phi : M \rightarrow M$ is a $C^\infty$-diffeomorphism, then $\underline{\dim}_{F}(\Phi_*\mu) = \underline{\dim}_{F}(\mu)$. 
Symetrically, the $C^k$-upper Fourier dimension test if, for any localization $\chi d\mu$ of $\mu$, there exists a $C^k$-chart $\varphi$ such that one has some decay for the Fourier transform of $\varphi_*(\chi d \mu)$. This quantity is also $C^\infty$-intrinsic. Still, beware that the upper and lower Fourier dimensions depends on the dimension of the ambiant manifold. 
\end{remark}

\begin{remark}
Let $E \subset M$ be a compact set lying in a manifold $M$ of dimension $d$. Fix a bump function $\chi$ and a local chart $\varphi \in \text{Chart}(\chi,C^k)$. For $\mu \in \mathcal{P}(E)$ a measure supported in $E \subset M$, we have the following bounds:
$$ 0 \leq \underline{\dim}_{F,C^k} \mu \leq \dim_F \varphi_*(\chi d\mu) \leq \overline{\dim}_{F,C^k} \mu .$$
Moreover, if $\dim_H E < d$, then:
$$ \overline{\dim}_{F,C^k} \mu \leq \dim_H E. $$
\end{remark}

\begin{example}
Let $M$ be a manifold of dimension $d$, and consider any smooth hypersurface $N \subset M$. Let $k \geq 1$. Let $\mu$ be any smooth and compactly supported measure on $N$. Then: $$ \underline{\dim}_{F,C^k} (\mu) = 0, \quad \overline{\dim}_{F,C^k} (\mu) = d-1.$$
The first fact is easily proved by noticing that, locally, $N$ is diffeomorphic to a linear subspace of $\mathbb{R}^d$, which has zero Fourier dimension. The second fact is proved by noticing that, locally, $N$ is diffeomorphic to a half sphere, and any smooth measure supported on the half sphere exhibit power decay of its Fourier transform with exponent $(d-1)/2$.
\end{example}

\begin{remark}
It seems that, for some well behaved measures $\mu \in \mathcal{P}(E)$ supported on compacts $E$ with $\dim_H E < d$, one might expect the quantity $\overline{\dim}_{F,C^k} \mu$ is be comparable to $\dim_H E$. For some measures lying in a 1-dimensionnal curve, this is the content of Theorem 2 in \cite{Ek16}.
\end{remark}

\begin{remark}
Using this langage, the results of \cite{BD17}, $\cite{LNP19}$, \cite{SS20}, \cite{Le21} and \cite{Li20}  all implies positivity of the lower Fourier dimension by $C^2$ charts of some measures (respectively: Patterson-Sullivan measures, Patterson-Sullivan measures, equilibrium states, equilibrium states, and stationary measures). This is a bit stronger than a related notion found in \cite{LNP19}, namely the \say{$C^2$-stable positivity of the Fourier dimension}. The results in our paper implies the following: under the conditions of Theorem 1.2, the equilibrium state $m_F \in \mathcal{P}(NW(\phi))$ satisfies $$\underline{\dim}_{F,C^2} (m_F) >0,$$
where the non-wandering set $NW(\phi)$ of the geodesic flow is seen in the unit tangent bundle $T^1 M$. In particular, its lower Fourier dimension is positive.
\end{remark}

\subsection{A variation with real valued phases}

For completeness, we suggest two variations for intrinsic notions of Fourier dimension for a measure in an abstract manifold. The first is exposed in this subsection, and the next will be discussed in the next subsection. 
Inspired by the computations made in section 4, we may want to look at more general oscillatory integrals involving $\mu$. A possibility is the following.

\begin{definition}
Let $M$ be a smooth manifold of dimension $d$. Let $E\subset M$ be a compact set. Let $\mu \in \mathcal{P}(E)$. Let $k \in \mathbb{N}^*$. Let $\text{Bump}(E)$ denote the set of all smooth functions $\chi : M \rightarrow \mathbb{R}$ such that $\text{supp}(\chi)$ is contained in a local chart. We denote by $\text{Phase}(\chi,C^k)$ the set of all real valued and $C^k$ maps $\psi : U \rightarrow \mathbb{R}$ with nonvanishing differential, where $U \supset \text{supp}(\chi)$ is an open set containing the support of $\chi$. Now, define the \underline{lower Fourier dimension} of $\mu$ by $C^k$ phases of $M$ by:
$$ \underline{\dim}_{F,C^k}^{\text{real}}(\mu) := \inf_{\chi \in \text{Bump}(E)} \inf \{ \dim_F( \psi_*(\chi d\mu) ) , \ \psi \in \text{Phase}(\chi,C^k) \}. $$
Similarly, define the \underline{upper Fourier dimension} of $\mu$ by $C^k$ phases of $M$ by:
$$ \overline{\dim}_{F,C^k}^{\text{real}}(\mu) := \inf_{\chi \in \text{Bump}(E)} \sup\{ \dim_F( \psi_*(\chi d\mu) ) , \ \psi \in \text{Phase}(\chi,C^k) \}. $$
As before, we also denote $\underline{\dim}_F^{\text{real}}(\mu) := \underline{\dim}_{F,C^\infty}^{\text{real}}(\mu)$.
\end{definition}

\begin{remark}
First of all, notice that $\psi_*(\chi d\mu)$ is a measure supported in $\mathbb{R}$, so its Fourier transform is a function from $\mathbb{R}$ to $\mathbb{C}$. More precisely:
$$ \forall t \in \mathbb{R}, \ \widehat{\psi_*(\chi d\mu)}(t) := \int_{E} e^{i t \psi(x)} \chi(x) d\mu(x) . $$
Like before, the lower/upper Fourier dimensions with real phases are $C^\infty$-intrinsic in the sense that for any $C^\infty$-diffeomorphism $\Phi : M \rightarrow M$, we have $\underline{\dim}_{F,C^k}^{\text{real}}(\Phi_*\mu) = \underline{\dim}_{F,C^k}^{\text{real}}(\mu) $ and $\overline{\dim}_{F,C^k}^{\text{real}}(\Phi_*\mu) = \overline{\dim}_{F,C^k}^{\text{real}}(\mu)$. 
\end{remark}

\begin{example}
Let $M$ be a smooth manifold, and let $N$ be a smooth submanifold of $M$. Let $\mu$ be a smooth and compactly supported probability measure in $N$. Then:
$$ \underline{\dim}_{F,C^k}^{\text{real}}(\mu) = 0 , \quad \overline{\dim}_{F,C^k}^{\text{real}}(\mu) = \infty. $$
These equalities can be proved as follow. Consider some smooth bump function $\chi$ with small enough support. Now, there exists a phase $\psi$, defined on a neighborhood $U$ of $\text{supp}(\chi)$, with nonvanishing differential on $U$ but which is constant on $N$. The associated oscillatory integral $\widehat{\psi_*(\chi d\mu)}$ doesn't decay, hence the computation on the lower Fourier dimension with real phases. There also exists smooth a phase $\psi$ such that $(d\psi)_{|TN}$ doesn't vanish. By the non-stationnary phase lemma, the associated oscillatory integral decay more than $t^{-N}$, for any $N \geq 0$, hence the computation on the upper Fourier dimension with real phases. \\
Notice how, in particular, $\min(\overline{\dim}_{F,C^k}^{real}(\mu),d)$ may be strictly larger than the Hausdorff dimension of the support of $\mu$. This may be a sign that this variation of the upper dimension isn't well behaving as a \say{Fourier dimension}. 
\end{example}

\begin{lemma}
We can compare this Fourier dimension with the previous one. We have:
$$ \underline{\dim}_{F,C^k}(\mu) \leq \underline{\dim}_{F,C^k}^{real}(\mu) , \quad \overline{\dim}_{F,C^k}(\mu) \leq \overline{\dim}_{F,C^k}^{real}(\mu) . $$

\end{lemma}

\begin{proof}
Let $\alpha<\underline{\dim}_{F,C^k}(\mu)$. Then, for any bump function $\chi$ and for any associated local chart $\varphi$, there exists some constant $C$ such that, for all $\xi \in \mathbb{R}^d \setminus \{0\}$, we have $|\widehat{\varphi_*(\chi d\mu)}(\xi)| \leq C |\xi|^{-\alpha/2}$. Now fix $\psi : U \rightarrow \mathbb{R}$ with nonvanishing differential, where $\text{supp}(\chi) \subset U$. By the submersion theorem, there exists a local chart $\varphi:U \rightarrow \mathbb{R}^d$ such that $\varphi(x) = \psi(x) e_1 + \sum_{j=2}^d f_j(x) e_j$ (where $(e_i)_i$ is the canonical basis of $\mathbb{R}^d$, and where $f_j$ are some real valued functions). Hence, one can write:
$$ |\psi_*(\chi d\mu)(t)| = |\varphi_*(\chi d\mu)(t e_1)| \leq C |t|^{-\alpha/2}.$$
Hence $\underline{\dim}_{F,C^k}^{real}(\mu) \geq \alpha$, and this for any $\alpha < \underline{\dim}_{F,C^k}(\mu)$, hence the desired bound. \\

The second bound is proved as follow. Let $\alpha < \overline{\dim}_{F,C^k}(\mu)$. Let $\chi$ be a small bump function. There there exists a local chart $\varphi : U \rightarrow \mathbb{R}^d$, with $U \supset \text{supp}(\chi)$, such that $\widehat{\varphi_*(\chi d\mu)} \lesssim |\xi|^{-\alpha/2}$. Let $u \in \mathbb{S}^{d-1}$ and consider $\psi(x) := u \cdot \varphi(x)$. It is easy to check that $\psi$ has nonvanishing differential, and since, for any $t \in \mathbb{R} \setminus \{0\}$,
$$ |\widehat{\psi_*(\chi d\mu)}(t)| = |\widehat{\varphi_*(\chi d\mu)}(ut)| \lesssim |t|^{-\alpha/2} ,$$
we get $\overline{\dim}_{F,C^k}^{real}(\mu) \geq \alpha$. The bound follow.
\end{proof}

In concrete cases, we expect the lower Fourier dimension and the lower Fourier dimension with real phases to be equal. Unfortunately, our choices of definitions doesn't clearly make that happen all the time. We have to add a very natural assumption for the equality to hold.

\begin{definition}
Let $\mu \in \mathcal{P}(E)$, where $E \subset M$ is a compact subset of a smooth manifold. We say that $\mu$ admits reasonnable constants for $C^k$-phases if, for any $\alpha < \underline{\dim}^{\text{real}}_{F,C^k}(\mu)$, and for any $\chi \in \text{Bump}(E)$, the following hold:
$$ \forall R \geq 1, \ \exists C_R \geq 1, \ \forall \psi \in \text{Phase}(\chi,C^k), $$ $$  \left( \|\psi\|_{C^k} + \sup_{x \in U} \| (d\psi)_x \|^{-1} 
 \leq R \right) \Longrightarrow \left( \forall t \in \mathbb{R}^*, \ |\widehat{\psi_*(\chi d\mu)}(t)| \leq C_R t^{-\alpha/2} \right). $$
\end{definition}

Under this natural assumption, we have equality of the lower Fourier dimensions.

\begin{lemma}
Let $\mu\in \mathcal{P}(E)$, where $E \subset M$ is a compact subset of some smooth manifold $M$. Suppose that $\mu$ admits reasonnable constants for $C^k$-phases. Then:
$$ \underline{\dim}_{F,C^k}(\mu) = \underline{\dim}_{F,C^k}^{real}(\mu) $$
\end{lemma}

\begin{proof}
An inequality is already known, we just have to prove the second one. The proof of the other inequality is the same argument as the one explained in Remark 4.6.
\end{proof}

\subsection{A directionnal variation}

A second natural and intrinsic idea would be to fix some (spatial) direction on which to look for Fourier decay. We quickly discuss these notions and then we will move on to discuss some notions of Fourier dimensions for sets.

\begin{definition}
Let $E \subset M$ be a compact set in some smooth manifold. Let $V \subset TM$ be a continuous vector bundle on an open neighborhood $\tilde{E}$ of $E$. Denote by $\text{Bump}^V(E)$ the set of all smooth bump functions with support included in $\tilde{E}$, and included in some local chart. For some $\chi \in \text{Bump}^V(E)$, denote by $\text{Phase}^V(\chi,C^k)$ the set of all $C^k$ maps $\psi:U \rightarrow \mathbb{R}$ such that $(d\psi)_{|V}$ doesn't vanish on $U$, where $\text{supp}(\chi) \subset U \subset \tilde{E}$ is some open set. \\

For $\mu \in \mathcal{P}(E)$, we define its \underline{lower Fourier dimension} in the direction $V$ for $C^k$ phases by:
$$ \underline{\dim}_{F,C^k}^V(\mu) := \inf_{\chi \in \text{Bump}^V(E)} \inf \{ \dim_F( \psi_*(\chi d\mu) ) , \ \psi \in \text{Phase}^V(\chi,C^k) \} .$$
Similarly, define its \underline{upper Fourier dimension} in the direction $V$ for $C^k$ phases by:
$$ \overline{\dim}_{F,C^k}^V(\mu) := \inf_{\chi \in \text{Bump}^V(E)} \sup \{ \dim_F( \psi_*(\chi d\mu) ) , \ \psi \in \text{Phase}^V(\chi,C^k) \} .$$
\end{definition}

\begin{remark}
Again, these notions of Fourier dimensions are $C^\infty$-intrinsic, in the following sense: if $\Phi:M \rightarrow M$ is a $C^k$-diffeomorphism of $M$, then $  \underline{\dim}_{F,C^k}^{\Phi_*V}(\Phi_*\mu) =  \underline{\dim}_{F,C^k}^V(\mu) $, and $ \overline{\dim}_{F,C^k}^{\Phi_*V}(\Phi_*\mu) =  \overline{\dim}_{F,C^k}^V(\mu)$.
\end{remark}

\begin{remark}
With these notations, the results found in \cite{Le23} implies that, for any \say{nonlinear} and sufficiently bunched solenoid $S$, and for any equilibrium state $\mu$, one has $\underline{\dim}_{F,C^{1+\alpha}}^{E_u}(\mu)>0$, where $E_u$ is the unstable line bundle associated to the dynamics on the solenoid.
\end{remark}

\begin{lemma}
Let $V_1,\dots,V_n \subset TM$ be some continuous vector bundles defined on some open neighborhood $\tilde{E}$ of $E$. Suppose that $(V_1)_p + \dots (V_n)_p = T_p M$ for all $p \in \tilde{E}$. Then:
$$ \min_{j} \underline{\dim}_{F,C^k}^{V_j}(\mu) = \underline{\dim}_{F,C^k}^{\text{real}}(\mu) , \quad \max_{j} \overline{\dim}_{F,C^k}^{V_j}(\mu) \leq \overline{\dim}_{F,C^k}^{\text{real}}(\mu) .$$ \end{lemma}

\begin{proof}
Let $\alpha < \underline{\dim}_{F,C^k}^{\text{real}}(\mu)$. Then, for any bump $\chi$ and associated phase $\psi$, one has $\widehat{\psi_*(\chi d\mu)}(t) \lesssim |t|^{-\alpha/2}$. In paticular, for any phase $\phi_j \in \text{Phase}^{V_j}(\chi,C^k)$, the previous decay holds, and so $\min_{j} \underline{\dim}_{F,C^k}^{V_j}(\mu) \geq \alpha$. Hence $\min_{j} \underline{\dim}_{F,C^k}^{V_j}(\mu)  \geq \underline{\dim}_{F,C^k}^{\text{real}}(\mu)$. \\

Now let $\alpha < \min_{j} \underline{\dim}_{F,C^k}^{V_j}(\mu)$. Then, for all $j$, for any bump $\chi$, and for any phase $\psi_j \in \text{Phase}^{V_j}(\chi,C^k)$, the previous decay applies. Now, if we fix some $\chi$ and some associated phase $\psi \in \text{Phase}(\chi,C^k)$, we know that at each point $p$, $(d\psi)_p$ is nonzero. In particular, there exists $j(p)$ such that $(d\psi)_{|V^{j(p)}_p} \neq 0$. Following the proof of Theorem 4.5, we can show by using a partition of unity that this implies $\widehat{\psi_*(\chi d\mu)}(t) \lesssim |t|^{-\alpha/2}$. Hence $\underline{\dim}_{F,C^k}^{\text{real}}(\mu) \geq \alpha$, and we have prove equality. \\

For our last bound, let $\alpha < \max_j \overline{\dim}_{F,C^k}^{V_j} \mu$. Then there exists $j$ such that, for all bump $\chi$, there exists an associated phase $\psi_j \in \text{Phase}^{V_j}(\chi,C^k)$ such that $\widehat{\psi_j(\chi d\mu)}(t) \lesssim |t|^{-\alpha/2}$. Since $\psi_j \in \text{Phase}(\chi,C^k)$, we get $\overline{\dim}_{F,C^k}^{\text{real}} \mu \geq  \max_j \overline{\dim}_{F,C^k}^{V_j} \mu$.
\end{proof}

\begin{remark}
The reverse bound for the upper dimensions is not clear: if for all bump functions $\chi$, there exists a phase $\psi$ with good fourier decay properties for $\mu$, then nothing allows us to think that $\psi$ is going to have nonvanishing diffenrential in some fixed $V_j$ on all $E$.
\end{remark}

\subsection{What about sets ?}

We finally define some intrinsic notions of Fourier dimensions for sets. First of all, recall that the usual definition for some $E \subset \mathbb{R}^d$ is:

$$ \dim_F(E) := \sup \{ \dim_F(\mu) \leq d \ , \ \mu \in \mathcal{P}(E) \} \leq \dim_H(E).$$

In particular, in view of the proof of Lemma A.1, we see that any measure $\mu$ with some Fourier decay properties may be localized anywhere on its support to still yield a measure with large Fourier dimension. Hence we find the following \say{localized} formula, for any $\varepsilon>0$:

 $$ \dim_F(E) = \underset{U \text{ open}}{\sup_{U \cap E \neq \emptyset}} \dim_F (E \cap U).$$ 

Now, we have two main ways to define the up(per) and low(er) Fourier dimension of a compact set in a manifold: directly computing the Fourier dimension of parts of $E$ in local charts, or taking the sup over the previously defined notions of Fourier dimension for measures. 

\begin{definition}
Let $E \subset M$ be a compact set included in some smooth manifold. We define its lower/upper Fourier dimension with $C^k$-charts by:
$$ \underline{\dim}_{F,C^k}(E) := \sup\{  \underline{\dim}_{F,C^k}(\mu)  \leq d, \ \mu \in \mathcal{P}(E)\}, \quad \overline{\dim}_{F,C^k}(E) := \sup\{  \overline{\dim}_{F,C^k}(\mu) \leq d, \ \mu \in \mathcal{P}(E)\}, $$
We also define the $C^k$-low Fourier dimension and $C^k$-up Fourier dimensions of $E$ by:
$$ \uwave{\dim}{}_{F,C^k}(E) := \underset{U \text{ open chart}}{\sup_{U \cap E \neq \emptyset}} \inf \{\dim_F (\varphi(E \cap U)) \ , \ \varphi:U \rightarrow \mathbb{R}^d \ C^k \text{ local chart} \}, $$
$$ {\owave{\dim}}_{F,C^k}(E) := \underset{U \text{ open chart}}{\sup_{U \cap E \neq \emptyset}} \sup \{\dim_F (\varphi(E \cap U)) \ , \ \varphi:U \rightarrow \mathbb{R}^d \ C^k \text{ local chart} \}. $$
\end{definition}

\begin{remark}
The low and up Fourier dimension are $C^k$-intrinsic in the natural sense. For exemple, if $\Phi : M \rightarrow M$ is a $C^k$-diffeomorphism, then $ \uwave{\dim}{}_{F,C^k}(\Phi(E)) =  \uwave{\dim}{}_{F,C^k}(E) $. The lower and upper Fourier dimension are $C^\infty$-intrinsic.
\end{remark}

\begin{lemma}
Let $E \subset M$ be a compact set in some smooth manifold $M$. Then:
$$ 0 \leq \underline{\dim}_{F,C^k}(E) \leq \uwave{\dim}{}_{F,C^k}(E) \leq {\owave{\dim}}{}_{F,C^k}(E) =  \overline{\dim}_{F,C^k}(E) \leq \dim_H(E) \leq d  .$$
\end{lemma}

\begin{proof}
Let us prove all the inequalities in order, from left to right. $ 0 \leq \underline{\dim}_{F,C^k}(E) $ is trivial. Let us prove the second one. \\

Let $\alpha<\underline{\dim}_{F,C^k}(E)$. By definition, there exists some probability measure $\mu \in \mathcal{P}(E)$ such that $\underline{\dim}_{F,C^k}(\mu) \geq \alpha$.
Now, since the support of $\mu$ is nonempty, there exists $U$ some small open set and a bump function $\chi$ supported in $U$ such that $\chi d\mu$ is a (localized) nonzero measure. Let $\varphi : U \rightarrow \mathbb{R}^d$ a local chart. Then, by hypothesis on $\mu$, $\dim_F \varphi_*(\chi d\mu) \geq \alpha$. In particular, since (up to normalization) $\varphi_*(\chi d\mu) \in \mathcal{P}(\varphi(E \cap U)) $, we have $\dim_F \varphi(E \cap U) \geq \alpha$. This for any local chart $\varphi$, and so $\inf_{\varphi} \dim_F(\varphi(E \cap U)) \geq \alpha$. This yields $\underline{\dim}_{F,C^k}(E) \geq \alpha$. Since this is true for any $\alpha<\underline{\dim}_{F,C^k}(E)$, we get the desired inequality. \\

The inequality $\uwave{\dim}{}_{F,C^k}(E) \leq {\owave{\dim}}{}_{F,C^k}(E) $ is trivial. Let us prove the equality between ${\owave{\dim}}{}_{F,C^k}(E)$ and $\overline{\dim}_{F,C^k}(E)$. Let $\alpha < {\owave{\dim}}{}_{F,C^k}(E)$. Then, there exists some small open set $U$ (such that $U \cap E \neq \emptyset$ and a local chart $\varphi:U \rightarrow \mathbb{R}^d$ such that $\dim_F(\varphi(U \cap E)) \geq \alpha$. By definition, it means that there exists some measure $\nu \in \mathcal{P}(\varphi(E \cap U))$ such that $\dim_F \nu \geq \alpha$. Letting $\mu := \varphi^{-1}_*\nu \in \mathcal{P}(E \cap U)$ yields a measure supported in $E$ that satisfies $\overline{\dim}_{F,C^k}(\mu) \geq \alpha$ (in view of the proof of Lemma A.1). Hence, $\overline{\dim}_{F,C^k}(E) \geq \alpha$. This, for any $\alpha < {\owave{\dim}}{}_{F,C^k}(E)$, so that $\overline{\dim}_{F,C^k}(E) \geq {\owave{\dim}}{}_{F,C^k}(E)$. \\

Let us prove the other inequality. Let $\alpha < \overline{\dim}_{F,C^k}(E) $. By definition, there exists $\mu \in \mathcal{P}(E)$ such that $\overline{\dim}_{F,C^k}(\mu) \geq \alpha$. Now let $U$ be some small open set with $\mu_{|U} \neq 0$, and let $\varphi:U \rightarrow \mathbb{R}^d$ be a local chart. Let $\chi$ be some bump function supported in $U$. Then, by hypothesis on $\mu$, we have $\dim_F \varphi_*(\chi d\mu) \geq \alpha$. In particular, $\dim_F( \varphi(E \cap U) ) \geq \alpha.$ Hence $\owave{\dim}_{F,C^k}(E) \geq \alpha$. This proves the other inequality, and hence concludes the proof that ${\owave{\dim}}{}_{F,C^k}(E) =  \overline{\dim}_{F,C^k}(E)$. \\

Finally, the fact that the Hausdorff dimension is invariant under $C^1$-diffeomorphisms implies
$$ \owave{\dim}_{F,C^k}(E) = \sup_{U} \sup_{\varphi} \dim_F(\varphi(U \cap E)) \leq  \sup_{U} \sup_{\varphi} \dim_H( \varphi(E \cap U)) = \sup_{U} \dim_H(E \cap U) = \dim_H(E) \leq d.$$\end{proof}

\begin{example}
Let $N \subset M$ be a hypersurface in some smooth manifold $M$. Then:
$$ \underline{\dim}_{F,C^k}(N) = \uwave{\dim}{}_{F,C^k}(N) = 0 \quad, \owave{\dim}_{F,C^k}(N) = \overline{\dim}_{F,C^k}(N) = \dim_H(N) = d-1 .$$
\end{example}

\begin{example}
We can finally state the result that we wanted to state. Let $M$ be a convex-cocompact hyperbolic surface. Let $NW(\phi) \subset T^1 M$ be the non-wandering set of the geodesic flow $\phi$, seen as lying in the unit tangent bundle of $M$. Then:
$$ \underline{\dim}_{F,C^2}(NW(\phi)) > 0. $$
\end{example}

\begin{example}

Let $L$ be a 1-dimensionnal manifold, and let $E \subset L$ be \underline{\textbf{any compact subset}}. Then:
$$ \overline{\dim}_{F,C^1} E = \dim_H E. $$
This very striking result is proved in \cite{Ek16}. Also, Ekstrom proves that, for any $k \geq 1$, we have $\overline{\dim}_{F,C^k} E \geq (\dim_H E)/k$. This motivates the following question: do we have, for any compact set $E$ in any manifold $M$, the formula $\overline{\dim}_{F,C^1}(E) = \dim_H(E) $ ? 
\end{example}

\begin{remark}
Other natural questions are the following. Can we find an example of set $E \subset \mathbb{R}^d$ such that $\underline{\dim}_{F,C^k}(E) < \uwave{\dim}{}_{F,C^k}(E)$ ? Or is it always an equality ? Is the lower Fourier dimension $C^k$-intrinsic ?
\end{remark}

For completeness, we conclude by introducing the real variation for the lower Fourier dimension. We will not introduce this variation for the upper Fourier dimension, as we said earlier that these seems to behave quite badly with respect to the Hausdorff dimension. To keep it concise, we will not discuss the directionnal variations.

\begin{definition}
Let $E \subset M$ be a compact subset of some smooth manifold $M$. Define the lower Fourier dimension with $C^k$-phases by:
$$ \underline{\dim}_{F,C^k}^{\text{real}}(E) := \sup \{ \underline{\dim}_{F,C^k}^{\text{real}} \mu \leq d \ , \ \mu \in \mathcal{P}(E) \}. $$
\end{definition}

\begin{remark}
By Lemma A.11, we see that $\underline{\dim}_{F,C^k}(E) \leq \underline{\dim}_{F,C^k}^{\text{real}}(E)$. Is this an equality, or are we able to produce an exemple were this inequality is strict ? A related question is: if we denote by $\mathcal{P}_{reas,C^k}(E)$ the set of probability measures that admits reasonnables constants for $C^k$-phases (see Definition A.12), do we have
$$ \underline{\dim}_{F,C^k}(E) = \sup\{ \underline{\dim}_{F,C^k} \mu \leq d  , \ \mu \in \mathcal{P}_{reas,C^k}(E) \} \quad ? $$
\end{remark}

\end{appendices}

\end{document}